\documentclass{amsart}
\usepackage{graphicx}
\usepackage[all]{xy}
\usepackage[ansinew]{inputenc}
\usepackage{amsmath}
\usepackage{amscd}
\usepackage{amssymb}
\usepackage{latexsym}
\usepackage[active]{srcltx}
\usepackage{mathrsfs}
\usepackage{color}
\usepackage{enumerate}
\usepackage{hyperref}

\newtheorem{thm}{Theorem}[section]

\theoremstyle{definition}
\newtheorem{cor}[thm]{Corollary}
\newtheorem{lem}[thm]{Lemma}
\newtheorem{prop}[thm]{Proposition}
\newtheorem{defn}[thm]{Definition}

\newtheorem{fact}[thm]{Fact}
\newtheorem{claim}{Claim}

\newtheorem*{thmA}{Theorem A}
\newtheorem*{thmB}{Theorem B}

\numberwithin{equation}{section}

\newcommand{\cl}[1]{\textbf{cl}_{#1}}

\newcommand{\N}{\mathbb{N}}
\newcommand{\Z}{\mathbb{Z}}
\newcommand{\Q}{\mathbb{Q}}
\newcommand{\R}{\mathbb{R}}

\newcommand{\tp}{\operatorname{tp}}

\newcommand{\dcl}{\operatorname{dcl}}

\newcommand{\dom}{\operatorname{dom}}

\newcommand{\Cal}{\mathcal}

\newcommand{\OR}{\overline{\R}}

\newcommand{\TT}{\tilde{T}}

\def \<{\langle}
\def \>{\rangle}

\def \((  {(\!(}
\def \)) {)\!)}

\def \cl {\mathrm{cl}}

\begin{document}

\title[A tame Cantor set]{A tame Cantor set}

\author[P. Hieronymi]{Philipp Hieronymi}
\address
{Department of Mathematics\\University of Illinois at Urbana-Champaign\\1409 West Green Street\\Urbana, IL 61801}
\email{phierony@illinois.edu}
\urladdr{http://www.math.uiuc.edu/\textasciitilde phierony}


\date{\today}

\begin{abstract}
A Cantor set is a non-empty, compact subset of $\R$ that has neither interior nor isolated points. In this paper a Cantor set $K\subseteq \R$ is constructed such that every set definable in $(\R,<,+,\cdot,K)$ is Borel. In addition, we prove quantifier-elimination and completeness results for $(\R,<,+,\cdot,K)$, making the set $K$ the first example of a modeltheoretically tame Cantor set. This answers questions raised by Friedman, Kurdyka, Miller and Speissegger. The work in this paper depends crucially on results about automata on infinite words, in particular B\"uchi's celebrated theorem on the monadic second-order theory of one successor and McNaughton's theorem on Muller automata, which have never been used in the setting of expansions of the real field.
\end{abstract}

\thanks{A version of this paper is to appear in the \emph{Journal of the European Mathematical Society}. The author was partially supported by NSF grant DMS-1300402 and by UIUC Campus Research Board award 14194.}

\maketitle

\section{Introduction} Let $\OR:=(\R,<,+,\cdot)$ denote the real ordered field. The results in this paper contribute to the research program of understanding  expansions of $\OR$ by constructible sets. A set is \textbf{constructible} if it is a finite boolean combination of open sets. The motivation behind this work is the following natural question which lies in the intersection of model theory and descriptive set theory\footnote{The question was first raised in \cite{FKMS} p. 1311.}:
\begin{center}
\emph{What can be said about sets definable in such an expansion in terms of the real projective hierarchy?}
\end{center}
\noindent As is well known, when expanding the real field by constructible sets, arbitrary complicated projective sets can happen to be definable. Indeed, every projective subset of $\R^n$ is definable in $(\OR,\N)$, see for example Kechris \cite[37.6]{Kechris}. However, there are many examples of expansions of $\OR$ whose definable sets are all constructible; among these structures are all o-minimal expansions of $\OR$ and several non-o-minimal ones (see \cite{Miller-tame,Miller-fast,Miller-iteration}). This paper aims to determine what kind of expansions lie between these two extremes. Surprisingly little is known. The primary result in this direction is due to Friedman, Kurdyka, Miller and Speissegger \cite{FKMS}. They construct a constructible set $E\subseteq [0,1]$ such that $(\OR,E)$ defines sets on every level of the projective hierarchy (that is for each $N\in \N$ there is a definable set in $\boldsymbol{\Sigma}_{N+1}^1\setminus \boldsymbol{\Sigma}_N^1$), but does not define every projective set. At the end of \cite{FKMS} the question is discussed whether there is a constructible set $K$ and $N\in \N$ such that $(\OR,K)$ defines non-constructible sets, yet every definable set is $\boldsymbol{\Sigma}_N^1$. In this paper, we will answer this question positively.

\begin{thmA} There is a constructible set $K\subseteq \R$ such that $(\OR,K)$ defines non-constructible sets, yet every definable set in $(\OR,K)$ is Borel.
\end{thmA}

\noindent This paper is also a contribution to the study of modeltheoretic tameness in expansions of the real field. Both sets $E$ from \cite{FKMS} and $K$ from this paper are Cantor sets. For our purposes, a \textbf{Cantor set} is a non-empty, compact subset of $\R$ that has neither interior nor isolated points. By Fornasiero, Hieronymi and Miller \cite{FHM} an expansion of $\OR$ does not define $\N$ (and hence not every projective set) if and only if every definable Cantor set of $\R$ is Minkowski null\footnote{A bounded set $A\subseteq \R^n$ is \textbf{Minkowski null} if $\lim_{r\to 0^+} r^{\varepsilon} N(A,r)=0$ for all $\varepsilon >0$, where $N(A,r)$ is minimum number of balls of radius $r$ needed to cover $A$.}. While prohibiting the existence of definable Cantor sets of positive Minkowski dimension in expansions that do not define $\N$, this result does not say much about definable sets in expansions that define Minkowski null Cantor sets. Again, the only result in this direction is the result from \cite{FKMS}, because the set $E$ is a Cantor set. The nondefinability of $\N$ in $(\OR,E)$ is deduced from the property that every subset of $\R$ definable in this expansion either has interior or is nowhere dense. While this statement can be interpreted as a weak form of topological tameness of the definable sets, it surely cannot be considered as tameness in terms of model theory. In fact, the structure $(\OR,E)$ defines a Borel isomorph of $(\OR,\N)$ and therefore does not satisfy any notion of what could reasonably be considered as modeltheoretic tameness. Just to give an example: its theory is obviously undecidable and there is no bound on the quantifier complexity needed to define all definable sets in this structure. This observation made Friedman and Miller ask the following question in personal communication with the author\footnote{Already at the end of \cite{FKMS} the question is raised whether there is a Cantor set different from $E$ such that more can be said about the definable sets in the expansion by that Cantor set.}: \begin{center}\emph{Is there a (modeltheorectically) tame Cantor set?} \end{center}
\noindent Here we give a positive answer to their question using the Cantor set $K$ from Theorem A. While it will follow easily from \cite{FKMS} that every bounded unary definable set in $(\OR,K)$ either has interior or is Minkowski null, we will say significantly more about the first-order theory of $(\OR,K)$ and definable sets in this structure. We will give a natural axiomatization of its theory (see Section \ref{section:theory} and Theorem \ref{thm:complete}) and prove a quantifier-elimination result in a suitably extended language (see Theorem \ref{thm:qe}). Since the precise axiomatization and quantifier-elimination results are technical, we postpone their statement.\newline

\noindent In \cite{FKMS} it is already pointed out that new ideas seem to be necessary to say more about the definable sets in expansions of $\OR$ by a Cantor set. This is indeed the case. In this paper we use some of the techniques from \cite{FKMS}, but we will have to develop several new tools to prove Theorem A and the existence of a tame Cantor set. Above all other we rely on a novel use of results about automata on infinite words. In particular, we recognize a deep connection between this research program and B\"uchi's famous theorem about the monadic second-order theory of one successor \cite{Buchi}. To the author's knowledge, this and related results have never been used for studying expansions of the real field. We regard this new relation between these research areas as one of the main contributions of this paper, and anticipate potential for further applications. We will outline some of these applications at the end of this introduction. First, we briefly describe how this connection arises.\newline

\noindent Many of the results in and around B\"uchi's paper are stated in terms of second-order logic and in terms of automata on infinite words, but all of them can be restated in terms of first-order model theory. Let $\Cal B$ be the two-sorted structure $(\N,\Cal P(\N),s_{\N},\in)$, where $s_{\N}$ is the successor function on $\N$ and $\in$ is the relation on $\N \times \Cal P(\N)$ such that $\in(t,X)$ iff $t \in X$. In \cite{Buchi} the decidability of the theory of $\Cal B$ and a quantifier-elimination result are established. The latter result, which is the most relevant to this paper, was later significantly strengthened by McNaughton \cite{McNaughton}. Here we will show that when a Cantor set $K$ is sufficiently regular, the expansion $(\OR,K)$ defines an isomorphic copy of $\Cal B$. And not only is such an isomorphic copy definable, we will see that for well chosen $K$ the complexity of the definable sets in $(\OR,K)$ is controlled by the complexity of the definable sets in $\Cal B$. Hence the results bounding the complexity of definable sets in $\Cal B$, such as the ones mentioned above, will bound the complexity of definable sets in $(\OR,K)$.\newline

\noindent Theorem A and the existence of a tame Cantor set are proved not only for expansions of the real field, but for a larger class of o-minimal expansions of $\OR$. An expansion $\Cal R$ of $\OR$ is \textbf{exponentially bounded} if for every function $f: \R \to \R$ definable in $\Cal R$ there exists $m\in \N$ such that $f$ is bounded at $+\infty$ by the $m$-the compositional iterate of $\exp$. All known o-minimal expansions of the real field are exponentially bounded.

\begin{thmB} There is a Cantor set $K\subseteq \R$ such that for every exponentially bounded o-minimal expansion $\Cal R$ of $\OR$, every definable set in $(\Cal R,K)$ is Borel.
\end{thmB}

\noindent  Here is a very rough outline of the proof. Following Cantor's classical construction we define a Cantor set $K$ by inductively removing middle `thirds' of a line segment. However, as in \cite{FKMS}, instead of always removing exactly a third of the previous segment, we remove increasingly larger and larger portions of the segments. This construction results in a Cantor set that is homeomorphic to the classical Cantor ternary set, but Minkowski null. Indeed it follows from results from \cite{FKMS} that every image of $K^n$ under functions definable in $\Cal R$ is Minkowski null. Let $Q$ denote the set of lengths of complementary intervals of $K$. Note that $Q$ is definable in $(\OR,K)$. We show that there is a set $\epsilon \subseteq Q \times K$ definable in $(\OR,K)$ such that the two-sorted structures $(Q,K,\epsilon,s_Q)$ and $(\N,\Cal P(\N),\in,s_{\N})$ are isomorphic, where $s_Q$ denotes the successor function on $(Q,<)$. Then we use known results about the latter structure to control the complexity of definable sets in $(Q,K,\epsilon,s_Q)$ and hence in $(\OR,K)$. Because $K$ is constructed to be very sparse, we are then able to show that the o-minimal structure does not induce new definable sets on $Q$ and $K$ other than the ones coming from $(Q,K,\epsilon,s_Q)$. This last step requires most of the technical work in this paper and involves a wide array of tools from o-minimality.\newline

\noindent Throughout this paper we assume familiarity with basic definitions and results in model theory, o-minimality and descriptive set theory. We refer to Marker \cite{Marker} for model theory, to van den Dries \cite{tametop} for o-minimality, to Kechris \cite{Kechris} for descriptive set theory. This paper aims to be self-contained with respect to ingredients from fractal geometry and from the theory of automata on infinite words. Nevertheless, a good reference on the former is Falconer \cite{Falconer} and on the latter is Khoussainov and Nerode \cite{automata}.

\subsection*{Remarks} We conclude this introduction with a few remarks about the optimality of the results and the applicability of the methods of this paper to other open questions.\newline

\noindent \textbf{1.} Because every Cantor set is interdefinable over $\OR$ with the set of midpoints of its complementary intervals, there is a discrete set $D\subseteq \R$ such that $(\OR,D)$ defines non-constructible sets, yet every definable set in $(\OR,D)$ is Borel. By \cite[Theorem B]{discrete2} we can even take $D$ to be closed and discrete.\newline

\noindent \textbf{2.} In \cite{FKMS} it was suggested that a Cantor set $K$ exists such that every definable set in $(\OR,K)$ is not only Borel, but even a boolean combination of $F_{\sigma}$ sets. We do not know whether or not the Cantor set constructed in this paper has this stronger property.\newline

\noindent \textbf{3.} Another question from \cite{FKMS} asks whether there is a constructible set $E\subseteq \R$ and $N\in \N$ such that every definable set in $(\OR,E)$ is $\boldsymbol{\Sigma}_N^1$ and $(\OR,E)$ defines a non-Borel set. We imagine that the ideas presented in this paper can be used to give a positive answer to this question. However, since $\Cal B$ does not define non-Borel sets, one has to replace the use of $\Cal B$ by the use of more expressive structures. For example, structures based on Rabin's work on the monadic second order theory of multiple successor \cite{Rabin} might prove useful here.\newline

\noindent \textbf{4.} An open question related to the optimality of \cite{FHM} is whether there is an expansion of $\OR$ that does not define $\N$, but defines both a Cantor set and a dense and codense set. The tools from \cite{FKMS} are known to be not enough to construct such an expansion. However, it seems reasonable to expect that the work in this paper can be adjusted to construct a Cantor set $K$, a dense and codense subset $X\subseteq \R$ such that $(\OR,K,X)$ not only does not define $\N$, but is modeltheoretically well-behaved. An amalgamation of the proofs from \cite{densepairs} and from this paper should yield this result.
\newline

\noindent\textbf{5.} A model theorist might ask what happens when we look at expansions of the ordered real additive group by Cantor sets. Although to the author's knowledge this was never stated explicitly in the literature, strong results can be deduced easily from known theorems. Consider the famous Cantor ternary set $C$. It is not Minkowski null (its Minkowski dimension is $\log_3(2)$). Therefore the theory of $(\OR,C)$ is undecidable by \cite{FHM}. The situation is very different when we replace the real field by the ordered real additive group. For $r\in \N_{>2}$, consider the expansion $\Cal T_r$ of $(\R,<,+,\Z)$ by a ternary predicate $V_r(x,u,k)$ that holds if and only if $u$ is a positive integer power of $r$, $k\in \{0,\dots,r-1\}$ and the digit of a base-$r$ representation of $x$ in the position corresponding to $u$ is $k$. As shown in Boigelot, Rassart and Wolper \cite{BRW}, it follows from B\"uchi's work that the theory of $\Cal T_r$ is decidable. Since $C$ is precisely the set of real numbers in $[0,1]$ in one of whose ternary expansion the digit 1 does not appear, $C$ is $\emptyset$-definable in $\Cal T_3$. Therefore the theory of $(\R,<,+,C)$ is decidable.

\subsection*{Acknowledgements} I thank Chris Miller for bringing this question to my attention. I also would like to thank Anush Tserunyan for answering my questions about descriptive set theory, Lou van den Dries for answering my questions about o-minimality and Carl Jockusch and Paul Schupp for answering my questions about the monadic second order theory of one successor.

\subsection*{Notations} Throughout, definable means definable with parameters. If we need to be specific about the language $\Cal L$ and the parameters $X$ used to define a set, we say this set is $\Cal L$-$X$-definable.
If we say that $\varphi$ is a $\Cal L$-formula, we mean that there are no additional parameters appearing in $\varphi$. For an arbitrary language $\Cal L$ and an arbitrary $\Cal L$-theory $T$, we denote the type of a tuple $z$ of elements of a model $M$ of $T$ over some subset $X$ of the universe of $M$ by $\tp_{\Cal L}(z|X)$. Whenever there is a second model of $N$ of $T$ and an $\Cal L$-embedding of $X$ into $N$, then we write $\beta \tp_{\Cal L} (z|X)$ for the $\Cal L$-type over $\beta(X)$ given by
\[
\{ \varphi(y,\beta(x_1),\dots,\beta(x_m)) \ : \ \varphi(y,x_1,\dots,x_n) \in \tp_{\Cal L}(z|X)  \}.
\]
We will sometimes drop the subscript $\Cal L$ when the language is clear from the context. The variable $i,j,k,m,n$ always range over $\N = \{0,1,2,\dots\}$. Given two sets $X, Y$, $Z\subseteq X\times Y$ and $x \in X$, we write $Z_x$ for $\{ y \in Y  : (x,y) \in Z\}$. We will use $\pi : X \times Y \to X$ for the projection onto the first factor. Moreover, if $X$ is ordered by $<$ and $x\in X$, we write $X_{\leq x}$ for $\{ z \in X : z \leq x\}$. We write $\pi$ for the projection of $Z$ onto $X$. Moreover, if $(Y,\prec)$ is a linear order such that every element except the minimum and maximum of $Y$ has a predecessor and successor, we denote the predecessor function on $Y$ by $p_Y$ and the successor function on $Y$ by $s_Y$. If $X$ is a subset of a topological space, we denote the closure of $X$ by $\cl(X)$. If $x\in \R^n$ and $\varepsilon \in \R_{> 0}$, we denote the ball of radius $\varepsilon$ around $x$ by $B_{\varepsilon}(x)$.
Moreover, if $X$ is linearly ordered by $<$, then we will also write $<$ for the lexicographic ordering on $X^n$ given by $<$. If $M$ is a real closed field, $c=(c_1,\dots,c_n)\in M^n$ and $q=(q_1,\dots, q_n)\in \Q^n$, we write $q \cdot c$ for $(q_1c_1,\dots,q_nc_n)$.

\section{Construction of $K$} \label{section:setup} Fix an o-minimal expansion $\Cal R$ of the real field $\OR$. We denote the language of $\Cal R$ by $\Cal L$ and the $\Cal L$-theory of $\Cal R$ by $T$. Throughout, we assume that $\Cal R$ is exponentially bounded. By combining Miller \cite{growth}, Speissegger \cite{pfaffian} and Lion, Miller and Speissegger \cite{lms} $(\Cal R, \exp)$ is an exponentially bounded o-minimal expansion of $\OR$ if the same holds for $\Cal R$. Since Theorem B holds for $\Cal R$ if it holds for $(\Cal R,\exp)$, we assume that $\Cal R$ defines $\exp$.\newline

\noindent We denote the $m$-th compositional iterate by $\exp_m$. Take an increasing sequence $(P_k)_{k \in \N}$ of finite subsets of $\Q$ such that $\bigcup_{k\in \N} P_k =\Q$. Set $q_0=1$. Now fix a sequence $(q_k)_{k\in\N_{>0}}$ of positive algebraically independent real numbers such that
 \begin{itemize}
 \item[(A)] $q_{k+1} > 3q_k$ for $k\in \N$,
 \item[(B)] $| \sum_{i=0}^{k} p_i q_i^{-1}| > q_{k+1}^{-1}$ for $k\in \N$ and $p_0,\dots,p_n \in P_k$ not all zero,
 \item[(C)] $\lim_{k\to +\infty} \exp_m(q_k)/q_{k+1} = 0$ for $m\in \N$.

 \end{itemize}
 We denote the range of this sequence by $Q$. Set $K_0:=[0,1]$ and for $i\geq 1$
\[
K_{i+1} := K_i \setminus \bigcup_c (c+q_{i+1}^{-1},c+q_i^{-1}-q_{i+1}^{-1}),
\]
where $c$ ranges over the right endpoints of the complementary intervals of $K_i$. Set $K := \bigcap_{i} K_i$. We fix this $Q$ and this $K$ for the rest of the paper. The construction of $Q$ and $K$ was already given at \cite[p.1320]{FKMS}\footnote{In \cite{FKMS} $Q$ is used to denote the set of reciprocals of our
$Q$.}. As is pointed out there, one can easily check that $K$ is a Cantor set and homeomorphic to the Cantor ternary set.

\subsection*{Monadic second order theory of one successor} The work in this paper depends crucially on well known results about the monadic second-order theory of one successor. Because we expect a significant portion of the readers to be unfamiliar with many of the results, we will review them here. For details and proofs we refer to \cite{automata}.\\

\noindent Let $\Cal B$ be the two-sorted structure $(\N,\Cal P(\N),s_{\N},\in)$, where $s_{\N}$ is the successor function on $\N$ and $\in$ is the relation on $\N \times \Cal P(\N)$ such that $\in(t,X)$ iff $t \in X$.
We denote the language of $\Cal B$ by $\Cal L_B$. The theory of $\Cal B$ is called the monadic second-order theory of one successor. In a landmark paper B\"uchi \cite{Buchi} showed that the theory of $\Cal B$ is decidable. He established the decidability of the monadic second order theory of one successor by proving that a subset of $\Cal P(\N)^n$ is definable in $\Cal B$ if and only if it is recognizable by what was later named a B\"uchi automaton. In this paper  we will use a substantial strengthening of B\"uchi's characterization of definability in $\Cal B$ due to McNaughton \cite{McNaughton}. This generalization states that a set is definable in $\Cal B$ if and only if it is recognizable by a deterministic Muller automaton. For the purposes of this paper we are not so much interested in what exactly a deterministic Muller automaton is, but rather in what this characterization tells us about the Borel complexity of any given subset of $\Cal P(\N)^n$ in $\Cal B$. Viewing $\Cal P(\N)$ as the product $\{0,1\}^{\N}$, we can endow $\Cal P(\N)$ with the topology that corresponds to the usual product topology on $\{0,1\}^{\N}$. Among other things the Borel complexity with respect to this topology of subsets of $\Cal P(\N)$ definable in $\Cal B$ was studied in B\"uchi and Landweber \cite{Landweber}. There the following result was deduced from McNaughton's Theorem.

\begin{fact}\cite[Corollary 1]{Landweber}\label{fact:landweber} Every subset of $\Cal P(\N)^n$ definable in $\Cal B$ is a boolean combination of $\boldsymbol{\Pi}_2^0$ and hence in $\boldsymbol{\Delta}_3^0$.
\end{fact}

\noindent We will also use easy facts about definability in $\Cal B$, such as the definability in $\Cal B$ of the usual order on $\N$ and the set of finite subsets of $\N$, which we denote by $\Cal P_{\operatorname{fin}}(\N)$.\newline

\noindent We now explain the connection to the topic of this paper. It is well known that the Cantor ternary set and $\Cal P(\N)$ are homeomorphic (see for example \cite[I.3.4]{Kechris} or \cite[6.9.1]{automata}). Since $K$ is constructed in almost exactly the same way as the Cantor ternary set, one can easily see that the same construction gives an homeomorphism $h$ between $K$ and $\Cal P(\N)$. From this construction, it is clear that this homeomorphism can be extended to an isomorphism between the two-sorted structures $(\N,\Cal P(\N),\in,s_{\N})$ and $(Q,K,\epsilon,s_Q)$ for a certain $\epsilon \subseteq Q \times K$. We will now remind the reader what the precise definitions of $h$ and  $\epsilon$ are.\newline

\noindent Recall that $Q$ was defined to be the range of a sequence $(q_i)_{i \in \N}$ of real numbers and that we defined $K_0:=[0,1]$ and for $i\geq 1$, $K_{i+1} := K_i \setminus \bigcup_c (c+q_{i+1}^{-1},c+q_i^{-1}-q_{i+1}^{-1})$,
where $c$ ranges over the right endpoints of the complementary intervals of $K_i$, and $K := \bigcap_{i} K_i$. Let $g : \N \to Q$ be the map taking $n$ to $q_n$ and let $h : \Cal P(\N) \to K$ map $X \subseteq \N$ to $\sum_{n \in X} (q_{n-1}^{-1} - q_n^{-1})$. We will leave it to the reader to check that $h$ is well-defined\footnote{This is really the same construction as for the Cantor ternary set. In the case of the Cantor ternary set, the set $Q$ is $3^{\N}$, and hence $q_n=3^{n}$. Thus $q_{n-1}^{-1} - q_n^{-1} = 3^{-n-1}- 3^{-n}= 2 \cdot 3^{-n}$. Since the Cantor ternary set is the set of all numbers between 0 and 1 who have a ternary expansions consisting only of 0's and 2's, one can see directly that a function defined in the same ways as $h$ maps into the Cantor ternary set.}. Let $R_n$ be the set of right endpoints of complementary intervals of $K_n$ and let $R$ be the set of right endpoints of complementary intervals of $K$. Define $e : Q \times K \to R$ to be the function that maps $(q_n,c)$ to the largest $r\in R_{n}$ with $r\leq c$. From the construction, we immediately get that $0 \leq c - e(q,c) \leq q^{-1}$ for every $c\in K$ and $q \in Q$. Let $\epsilon \subseteq  Q \times K$ be the set of all $(q_n,c)$ such that $e(q_n,c)\in R_{n}\setminus R_{n-1}$.

\begin{prop}\label{prop:isokb} The map $\beta=(g,h)$ is an isomorphism between the two-sorted structures $(\N,\Cal P(\N),\in,s_{\N})$ and $(Q,K,\epsilon,s_Q)$ and $h$ is a homeomorphism.
\end{prop}
\begin{proof} Checking that $h$ is a homeomorphism and $\beta$ is an isomorphism between the two structures, is routine and we leave the details to the reader.
\end{proof}

\noindent We can deduce immediately from the definition of $h$ that if $c=\sum_{n \in X} (q_{n-1}^{-1} - q_n^{-1})$ for some $c\in K$ and $X\subseteq \N$, then $e(q_m,c) = \sum_{n \in X, n\leq m} (q_{n-1}^{-1} - q_n^{-1})$.\newline



\noindent It is now a good time to point out why we need to use these results about $\Cal B$. Observe that $(\OR,K)$ defines the discrete set $D$ of midpoints of complementary intervals of $K$ and a map $f:D\to K$ that maps an element $d\in D$ to $\sup K\cap (-\infty,d]$ (see \cite[Proof of Theorem]{FHM}). The image $f(D)$ is dense in $K$. The complexity of the definable sets in $(\OR,K)$ can be seen as a direct consequence of the definability of $f$ (for evidence see \cite[Theorem 1.1]{discrete} and Hieronymi and Tychonievich \cite[Theorem A]{HT}). Most of the technical work in this paper, and in particular the use of results about $\Cal B$, will be towards controlling this map.

\subsection*{Definable sets in $(\OR,K)$}

In this section we will study the Cantor set $K$ and the discrete set $Q$ in more detail.  The goal is to show that $Q$ and $\epsilon\subseteq Q\times K$ from the previous section are $\emptyset$-definable in $(\OR,K)$. It then follows from Proposition \ref{prop:isokb} that $(\OR,K)$ defines an isomorphic copy of $(\N,\Cal P(\N),\in,s_{\N})$.\newline

\noindent First, we define $L\subseteq [0,1]$ to be the left endpoints of complementary intervals of $K$. It is easy to see that both $L$ and $R$ are $\emptyset$-definable in $(\OR,K)$. Note also that elements in $L$ (or $R$) are left (resp. right) endpoints of complementary intervals of $K_n$ for some $n$.  We denote the set of left endpoints by $L_n$.

\begin{defn} Let $v: (L\cup R)\setminus \{0,1\}  \to \R$ map $d\in (L\cup R)\setminus \{0,1\}$ to the length of the complementary interval of $K$ one of whose endpoint $d$ is.
\end{defn}

\noindent Note that $v$ is $\emptyset$-definable in $(\OR,K)$. From the construction of $K$, in particular from $3q_n < q_{n+1}$, we directly get the following Lemma.

\begin{lem} \label{lem:v} Let $d \in R$. Then $v(d) = q_{n-1}^{-1} - 2q_n^{-1}$ if and only if $d\in R_n \setminus R_{n-1}$. The same statement holds with $R$ replaced by $L$ and $R_n$ by $L_n$.
\end{lem}

\noindent Hence for every $n\in \N$, both $L_n$ and $R_n$ are $\emptyset$-definable in $(\OR,K)$.
\begin{defn} Let $w: R \to L$ map $r \in R$ to the smallest $l\in L$ with $l>r$ and $v(l)\geq v(r)$ if such $l$ exists, and to $1$ otherwise.
\end{defn}

\noindent Since $v$ is definable in $(\OR,K)$, so is $w$.

\begin{cor} The set $Q$ is equal to $\{ w(r)-r \ : \ r \in R\}$ and hence is $\emptyset$-definable in $(\OR,K)$.
\end{cor}
\begin{proof}  Let $r\in R_n\setminus R_{n-1}$. By Lemma \ref{lem:v}, $w(r)$ is the smallest left endpoint of the complementary interval of $K_n$ larger than $r$. It follows immediately from the construction of $K_n$ that $w(r)-r = q_n^{-1}$. Hence $Q=\{ w(r)-r \ : \ r \in R\}$ and $Q$ is $\emptyset$-definable in $(\OR,K)$, since $w$ is.
\end{proof}

\noindent Remember that the predecessor function on $Q$ is denoted by $p_Q$. Since $Q$ is $\emptyset$-definable in $(\OR,K)$, so is $p_Q$.
For $q \in Q$, we set $R_q:= \{ r \in R \ : \ v(r)\geq p_Q(q)^{-1} - 2 q^{-1}\}$. By Lemma \ref{lem:v}, if $q=q_n$ for some $n$, then $R_q=R_n$. We immediately get $\emptyset$-definability of $e$ and $\epsilon$ in $(\OR,K)$.

\begin{cor} The function $e$ and the set $\epsilon$ are $\emptyset$-definable in $(\OR,K)$.
\end{cor}

\section{Preliminaries from o-minimality}

\subsection*{O-minimal structures} Throughout this paper, the reader is assumed to know the basic results about o-minimal structures and theories, as can be found in \cite{tametop}. The only o-minimal theory we will consider is the $\Cal L$-theory $T$. Let $M\models T$. For a subset $X\subseteq M$, we denote the \textbf{definable closure of $X$} in $M$ by $\dcl_{\Cal L}(X)$. When it is clear from the context which language $\Cal L$ is used, we simply write $\dcl(X)$. As is well known, in an o-minimal structure the definable closure operator is a pregeometry. We will make use of this fact routinely throughout the paper without further mentioning it.\newline

\noindent Every complete o-minimal theory expanding the theory of real closed fields has definable Skolem functions. Hence by extending the language $\Cal L$ and the theory $T$ by definitions, we may assume that $T$ has quantifier elimination and is universally axiomatizable, and that $\Cal L$ has no relation symbol other than $<$.

\subsection*{Limit points of images of $K^n$ under $\Cal L$-definable functions} We now recall some definitions and results from \cite{FKMS}. For details and proofs, the interested reader should consult the original source.\newline

\noindent  Define $\psi : [0,\infty) \to \R$ by
\[
\psi(t):= \left\{
          \begin{array}{ll}
            0, & \hbox{$t=0$;} \\
            e^{-1/t}, & \hbox{$0<t<1$;} \\
            t-1+e^{-1}, & \hbox{$t\geq 1$.}
          \end{array}
        \right.
\]
Note that $\psi$ is $\emptyset$-definable in $\Cal R$. For $m\in \Z$, we denote the $m$-th compositional iterate of $\psi$ by $\psi_m$. Note that $\lim_{k\to +\infty} \psi_m(q_k^{-1})/q_{k+1} = 0$ for every $m \in \N$. \newline

\noindent For $n\in \N$ and $l \in \Z$ define
\[
S_{n,l} = \{ x \in \R^n \ : \ 0 < x_n < \psi_l(x_{n-1}) < \dots < \psi_{(n-1)l}(x_1)\}.
\]
Again, note that every $S_{n,l}$ is $\emptyset$-definable in $\Cal R$. Let $\Cal T_n$ be the group of symmetries, regarded as linear transformation from $\R^n$ to $\R^n$ of the polyhedron inscribed in the unit ball of $\R^n$ whose vertices are the intersection of the unit sphere in $\R^n$ with $\{ tu \ : \ t > 0 \wedge u\in \{-1,0,1\}^n\}$.

\begin{fact}(cp. \cite[1.8]{FKMS}) \label{fact:1} Let $X \subseteq \R^n$ and $f: X \to [0,1]$ be definable in $\Cal R$. Then there is $j\in \N$ such that for all $y\in \cl(X)$ there is a $\delta \in \R_{>0}$ such that for every $m\leq n$ and for every $T \in \Cal T_n$ the restriction of $f$ to $B_{\delta}(y) \cap X \cap (y + T(S_{m,j}\times \{0\}^{n-m}))$ is continuous and extends continuously to the closure.
\end{fact}
\noindent Note that $j$ is chosen uniformly for all $y \in \cl(X)$. In \cite[1.8]{FKMS} $j$ could apriori depend on $y$. However, as stated directly after \cite[1.8]{FKMS}, the statement of \cite[1.8]{FKMS} holds in general exponentially bounded o-minimal structures, in particular in any elementary extensions of $\Cal R$. Thus by compactness one can easily check that there is a $j\in \N$ such that the conclusion of \cite[1.8]{FKMS} holds for this $j$ for every $y \in \cl(X)$.

\begin{fact}(cp. \cite[(ii) on p.1320]{FKMS})\label{fact:2}
For every $j>0$ there is $\delta >0$ such that for all $n>0$
\[
K^n - K^n \cap (-\delta,\delta)^n \subseteq \bigcup_{m\leq n} \bigcup_{T \in \Cal T_n} T(S_{m,j}\times \{0\}^{n-m}).
\]
\end{fact}
\noindent Fact \ref{fact:2} was proved in \cite{FKMS} not for $K$, but for a $\Q$-linearly independent subset of $K$. One can check that the proof only needs minor adjustment to give Fact \ref{fact:2}.
The next result we want to state is also only shown in \cite{FKMS} for a $\Q$-linearly independent subset of $K$. However, given Fact \ref{fact:2} the same proof works for $K$.

\begin{fact}(cp. \cite[Proof of Theorem on p.1319]{FKMS}) Every bounded unary definable set in $(\Cal R,K)$ either has interior or is Minkowski null.
\end{fact}

\noindent We will collect a few easy corollaries of Fact \ref{fact:2}. Since these results were not stated in \cite{FKMS}, not even for $\Q$-linearly independent subsets of $K$, we will give proofs here.

\begin{lem}\label{lem:limitg} Let $X\subseteq \R^{l+n}$ and $f : X \to [0,1]$ be $\Cal L$-$\emptyset$-definable  and continuous. Then there are $\Cal L$-$\emptyset$-definable functions $g_1,\dots,g_k : \R^{l+n} \to \R$ such that for every $x \in \pi(X)$ for every $c \in K^n\cap \cl(X_x)$ for every $\varepsilon > 0$ there is $\delta >0$ such that
for $c' \in K^n \cap B_{\delta}(c)\cap X_x$
\[
f(x,c') \in B_{\varepsilon}(g_1(x,c)) \cup \dots \cup B_{\varepsilon}(g_k(x,c)).
\]
\end{lem}
\begin{proof} Let $j>0$ as given by Fact \ref{fact:1}. Define $g_{m,T} : \R^{l+n} \to \R$ to be the function that maps $(x,y)$ to $
\lim_{z \to y} f(x,z)$, where $z$ ranges  over $y+\bigcup_{T \in \Cal T_n} T(S_{m,j}\times \{0\}^{n-m})\cap X_x$, if $y \in \cl(X_x)$ and $y+\bigcup_{T \in \Cal T_n} T(S_{m,j}\times \{0\}^{n-m})\cap X_x\neq \emptyset$, and to $0$ otherwise. By our choice of $j$, the function $g_{m,T}$ is well-defined and for every $x\in \R^l$ the function $g_{m,T}(x,-)$ extends $f(x,-)$ because of the continuity of $f$. Now take $x \in \R^l$, $c \in K^n$ and $\varepsilon > 0$. By Fact \ref{fact:2} we get $\delta>0$ small enough that
\[
K^n - c \cap (-\delta,\delta)^n \subseteq \bigcup_{m\leq n} \bigcup_{T \in \Cal T_n} T(S_{m,j}\times \{0\}^{n-m}).
\]
By Fact \ref{fact:1} we can reduce $\delta$ such that for every $m,T$ the function $f$ is continuous on $B_{\delta}(c) \cap \big(c+\bigcup_{T \in \Cal T_n} T(S_{m,j}\times \{0\}^{n-m})\cap X_x\big)$ and extends continuously to the closure. Hence by further reducing $\delta$, we have that for every $c' \in K^n \cap B_{\delta}(c)\cap \cl(X_x)$
\[
f(x,c') \in \bigcup_{m\leq n} \bigcup_{T \in \Cal T_n} B_{\varepsilon}(g_{m,T}(x,c)).
\]
\end{proof}

\begin{cor}\label{cor:limitpoints} Let $X\subseteq \R^{l+n}$ and $f : X \to [0,1]$ be $\Cal L$-$\emptyset$-definable and continuous. Let $g_1,\dots,g_k$ be given as in Lemma \ref{lem:limitg}.
Then
\[
\cl(f(x,X_x\cap K^n)) \subseteq \bigcup_{i=1}^k g_i(x,\cl(X_x)\cap K^n).
\]
For every $x \in \pi(X)$ and every $z\in \cl(f(x,X_x\cap K^n))$, the set
\[
\{ c \in K^n\cap \cl(X_x) \ : \ z=g_i(x,c) \hbox{ for some } i=1,\dots, k\}
\]
is closed.
%
\end{cor}
\begin{proof} Let $z\in \cl(f(x,X_x\cap K^n))$. Let $(c_j)_{j\in \N}$ be a sequence of elements in $K^n\cap X_x$ such that $\lim_{j \to \infty} f(x,c_j) =z$.
Since $K$ is bounded, we can assume $(c_j)_{j\in \N}$ converges. Since $\cl(X_x)\cap K^n$ is closed, there is $c \in \cl(X_x)\cap K^n$ such that $\lim_{j\to \infty} c_j = c$.
By Lemma \ref{lem:limitg}, $\lim_{j \to \infty} f(x,c_j) = g_i(x,c)$ for some $i\in \{1,\dots k\}$.\newline

\noindent For the second statement, let $x \in \pi(X)$, $z\in \cl(f(x,X_x\cap K^n))$ and suppose there is $c \in K^n \cap \cl(X_x)$ such that $z\neq g_i(x,c)$ for all $i=1,\dots,k$. Let $\varepsilon > 0$ be such that
$2\varepsilon < \min_{i=1,\dots,k} |g_i(x,c) - z|$. By Lemma \ref{lem:limitg} there is $\delta > 0$ such that $|f(x,c') - g_i(x,c)| < \varepsilon$ for all $c'\in K^n \cap X_x\cap B_{\delta}(c)$ and $i=1,\dots,k$.
Let $d \in K^n \cap \cl(X_x)\cap B_{\delta/2}(c)$. By Lemma \ref{lem:limitg} there is $\gamma > 0$ such that $\gamma < \delta/2$ and $|f(x,c') - g_i(x,d)| < \varepsilon$ for all $c'\in K^n \cap B_{\gamma}(c)\cap X_x$ and $i=1,\dots,k$. Let $c'\in K^n \cap X_x\cap B_{\gamma}(d)$. Since $\gamma < \delta/2$, $c' \in B_{\delta}(c)$. Then
\[
|g_i(x,d) -g_i(x,c)| \leq |g_i(x,d) - f(x,c')| + |g_i(x,c) - f(x,c')| < 2\varepsilon = |g_i(x,c) - z|.
\]
Thus $g_i(x,d) \neq z$ for $i=1,\dots,k$ and for all $d\in K^n \cap \cl(X_x)\cap B_{\delta/2}(c)$. Therefore the set in the second statement of the Corollary is closed.
\end{proof}

\subsection*{T-levels and T-convexity} We will now recall some less well known results about o-minimal theories. We start by a review of the notion of $T$-levels as introduced by Tyne \cite{Tyne}. For more details and proofs we refer to \cite{Tyne} and \cite{TyneJSL}. For this section, fix a model $M$ of $T$.

\begin{defn} Let $x \in M$. We write $0 \ll x$ if $x$ is greater than every element of $\dcl(\emptyset)$. For $0 \ll x \in M$, the \textbf{$T$-level of $x$}, denoted by $[x]$, is the convex hull in $M$ of the set of all values $f(x)$, with $f$ ranging over all $\Cal L$-$\emptyset$-definable strictly increasing and unbounded from above functions $f: M \to M$.
\end{defn}

\begin{fact}\cite[Corollary 3.11]{Tyne} Let $X\subseteq M$, $a \in M$ such that $[a] \cap X = \emptyset$. Then
\[
\{ x \in \dcl (X \cup \{a\}) \ : \ 0 \ll x \} = \bigcup_{0\ll x \in X} [x] \cup [a].
\]
\end{fact}

\noindent We will need the following generalization which can easily be deduced by induction on the size of $A$.

\begin{fact}\label{fact:tlevelfunction} Let $X\subseteq M$ and let $A\subseteq M$ be finite such that for all $a,b \in A$, $0\ll a$, $a\notin \bigcup_{0\ll x\in X} [x]$ and $[a]\neq [b]$ whenever $a\neq b$.
Then
\[
\{ x \in \dcl (X \cup A) \ : \ 0 \ll x \}  = \bigcup_{0\ll x \in X} [x] \cup \bigcup_{a \in A} [a].
\]
\end{fact}

\noindent Throughout this paper we will use the fact that given an elementary substructure of $M$ the $\Cal L$-type of a tuple of elements of $M$ over $X$ whose pairwise disjoint $T$-level do not intersect with $X$, is determined just by the order of the elements in the tuple. The next Lemma makes this statement precise.

\begin{lem}\label{lem:leveltype} Let $X\preceq M$ and let $a=(a_1,\dots,a_n),b=(b_1,\dots,b_n) \in M^n$ be such that
\begin{itemize}
\item[(i)] $0\ll a_1 < \dots < a_n$ and $0\ll b_1 < \dots < b_n$,
\item[(ii)] $a_i\notin \bigcup_{0\ll x\in X} [x]$ and $b_i\notin \bigcup_{0\ll x\in X} [x]$ for $i=1,\dots,n$,
\item[(iii)] $[a_i] \neq [a_j]$ and $[b_i] \neq [b_j]$ and
\item[(iv)] $\tp_{\Cal L}(a_i|X) = \tp_{\Cal L}(b_i|X)$.
\end{itemize}
Then $\tp_{\Cal L}(a|X) = \tp_{\Cal L}(b|X)$.
\end{lem}
\begin{proof} We prove the statement by induction on $n$. For $n=1$, the statement follows immediately from (iv). Now suppose the statement holds for $n-1$. Hence
\begin{equation}\label{eq:sametype}
\tp_{\Cal L} (a_1,\dots, a_{n-1}|X) = \tp_{\Cal L} (b_1,\dots, b_{n-1}|X).
\end{equation}
Let $f: M^{n-1} \to M$ be a $\Cal L$-$X$-definable function. We will write $a'$ for $(a_1,\dots, a_{n-1})$ and $b'$ for $(b_1,\dots,b_{n-1})$.
In order to show that $\tp_{\Cal L} (a_n|X,a') = \tp_{\Cal L} (b_n|X,b')$, it is enough to show that $f(a') \neq a_n$ and $f(b') \neq b_n$ and $f(a') < a_n  \hbox{ iff }  f(b')<b_n$.
By \eqref{eq:sametype} $0 \ll f(a')$ iff $0 \ll f(b')$. Since $0\ll a_n$ and $0\ll b_n$ by (ii), we can assume that $0\ll f(a')$ and  $0\ll f(b')$.
Then by Fact \ref{fact:tlevelfunction}, $f(a') \in  \bigcup_{i=1}^{n-1} [a_j] \cup \bigcup_{0\ll x\in X} [x]$,  and $f(b') \in  \bigcup_{i=1}^{n-1} [b_j] \cup \bigcup_{0\ll x\in X} [x]$.
Since $a_n\notin  \bigcup_{i=1}^{n-1} [a_j] \cup \bigcup_{0\ll x\in X} [x]$ and $b_n \notin \bigcup_{i=1}^{n-1} [b_j] \cup \bigcup_{0\ll x\in X} [x]$ by (ii) and (iii),
we have that $f(a')\neq a_n$ and $f(b')\neq b_n$. It is left to establish that $f(a') < a_n$  iff  $f(b')<b_n$. First suppose there is $j \in \{1,\dots,n\}$ such that $[f(a')] = [a_j]$. By \eqref{eq:sametype}, $[f(b')] = [b_j]$. Hence by (iii),
$f(a') < a_n$  iff $a_j < a_n$  iff $b_j < b_n$ iff $f(b') < b_n$. Now suppose there is $x \in X$ such that $[f(a')] = [x]$.  By \eqref{eq:sametype}, we have $[f(b')] = [x]$. By (ii) and (iv), $f(a') < a_n$ iff $x < a_n$ iff $x < b_n$ iff $f(b') < b_n$.
\end{proof}

\noindent We now turn our attention to the notion of a $T$-convex subring, which was introduced by van den Dries and Lewenberg \cite{LouLew}. For more details and proofs we refer the reader to \cite{LouLew} and its companion paper \cite{Lou}.

\begin{defn} A convex subring $V$ of $M$ is called \textbf{$T$-convex} if $f(V)\subseteq V$ for all $\Cal L$-$\emptyset$-definable functions $f : M \to M$.
\end{defn}

\begin{lem}\label{lem:tconvex} Let $a \in M$ and let  $U_a := \dcl(\emptyset) \cup \bigcup_{x \in M,[x] < [a]} [x]$. Then $V_a = U_a \cup -U_a$ is a $T$-convex subring. Its maximal ideal $\mathfrak{m}_a$ is $\{ x \in V_a \ : \ [|x|^{-1}] \geq [a]\}$.
\end{lem}
\begin{proof} The fact that $V_a$ is $T$-convex follows immediately from the definition of $V_a$ and \cite[Lemma 10.2]{Tyne}. Since $\mathfrak{m}_a$ is exactly the set of non-units of $V_a$, the description of $\mathfrak{m}_a$ in the Lemma holds. \end{proof}

\noindent For rest of this section fix $a\in A$ with $0\ll a$. We will introduce the following abbreviation which we will use throughout the paper. For $x \in M$, we denote the residue class of $x$ mod $\mathfrak{m}_a$ by $\overline{x}^a$. For a subset $X\subseteq M$, we write $\overline{X}^a$ for $\{ \overline{x}^a \ : \ x \in X\}$. By \cite[Remark 2.16]{LouLew} $\overline{V_a}^a$ expands naturally into model of $T$ that by \cite[Remark 2.11]{LouLew} is isomorphic to a tame\footnote{Tame here means tame in sense of \cite[p.76]{LouLew}.} substructure of $\Cal R$. For the rest of this paper, we will always consider $\overline{V_a}^a$ as a model $T$ in this way. Let $y\in V_a$ and $X\subseteq V_a$. When say $\overline{y}^a$ is $\dcl$-dependent over $\overline{X}^a$, we mean $\dcl$-dependent with respect to this $T$-model on $\overline{V_a}^a$. The following Lemma can easily be derived from \cite[Proposition 1.7]{Lou}.

\begin{lem}\label{lem:overlinedep} Let $x\in V_a^m, y \in V_a$. Then the following are equivalent:
\begin{itemize}
\item[(i)] there is a continuous $\Cal L$-$\emptyset$-definable $f: U\subseteq M^m \to M$ such that $x\in U$, $U$ is open and $[|f(x) - y|^{-1}]\geq [a]$,
\item[(ii)] $\overline{y}^a$ is $\dcl$-dependent over $\overline{x}^a$.
\end{itemize}
\end{lem}

\section{The theory and its consequences}

\label{section:theory}
In this section, we begin the study of the first-order theory of $(\Cal R,K,Q,\epsilon)$. We will define a theory $\TT$ in the language of the structure. In addition to deriving first consequences of this theory, we will prove that this structure is a model of $\TT$. The completeness of $\TT$ will be established later. We assume that the language $\Cal L$ of the underlying o-minimal structure $\Cal R$ already contains constant symbols for each element of $Q$ and for each element of $C$. From the construction of $Q$, we have that for every $n\in \N$ there is $m\in \N$ such that $s_{Q}(p)> \exp_n(p)$ for all $p\in Q$ with $p\geq q_{m}$. We denote the minimal such $m$ by $\omega(n)$. For $p\in \Q^n$ we denote by $\xi(p)$ the minimal $k\in \N$ such that
\[
\{\ \frac{s \cdot p}{2 (t \cdot p)+1}  \ : \ s,t \in \{-1,0,1\}^n,  2 (t \cdot p)\neq -1 \ \} \subseteq P_{k-1}.
\]

\subsection*{Notations and conventions} We consider structures
\[
\Cal M := (M,C,A,E),
\]
where $M$ is an $\Cal L$-structure, $C, A\subseteq M$ and $E \subseteq M^2$. Let $\Cal L_C$ be the language of this structure.\newline

\noindent Let $\Cal M$ be such a $\Cal L_C$-structure. With the usual abuse of notation, we will write $Q$ for the set of the interpretations of the constant symbols corresponding to elements in $Q$ and we will write $K$ for the set of the interpretations of the constant symbols corresponding to elements in $K$. Similarly for $q \in Q$ and $c\in K$, we will use $q$ for the interpretation of the constant corresponding to $q$ in $M$ and we will use $c$ for the interpretation of the constant corresponding to $c$ in $M$.\newline

\noindent In the following we want to restrict ourselves to $\Cal L_C$-structures that satisfy a certain $\Cal L_C$-theory. Before we can state this theory, we will have to introduce some further notations.

\begin{defn}\label{def:defofe} For every $c \in C$, let $S(c)$ be the set $\{ a \in A \ : \ E(a,c)\}$. Let $e : A \times C \to C$ be the function that maps $(a,c)$ to the unique $d \in C$ such that $S(d) = S(c)_{\leq a}$ if such $d$ exists, and to $0$ otherwise.
\end{defn}
\noindent One of the sentences in the $\Cal L_C$-theory we are going to define, will guarantee that the unique $d$ in the definition of $e$ will indeed always exist. We also introduce the following abbreviation: If $c=(c_1,\dots,c_n)\in C^n$ and $a \in A$, then we set $e(a,c):=\big(e(a,c_1),\dots,e(a,c_n)\big)$.

\begin{defn} For $a\in A$ and $c\in C$, we set
\[
\delta_{a,c} := \left\{
                  \begin{array}{ll}
                    1, & \hbox{if $E(a,c)$;} \\
                    0, & \hbox{otherwise.}
                  \end{array}
                \right.
\]
For $q=(q_1,\dots, q_n) \in \Q^n$ define $\mu_q : C^n \to A$ to be the function that maps $c=(c_1,\dots,c_n)$ to the minimum $a\in A$ such that $\sum_{i=1}^{n} q_i \delta_{a,c_i} \neq 0$, if such $a$ exists and $0$ otherwise.
\end{defn}

\noindent We remind the reader that given an order set $(Y,\prec)$, we denote the predecessor function on $Y$ by $p_Y$ and the successor function on $Y$ by $s_Y$ if such functions are well-defined. So whenever, $A$
is a closed and discrete subset of $M$, then $p_A$ and $s_A$ will denote the predecessor function and the successor function on $A$ with respect to $<$.

\subsection*{The theory} We are now ready to define the desired $\Cal L_C$-theory.
\begin{defn} Let $\TT$ be the $\Cal L_C$-theory consisting of the first-order $\Cal L_C$-sentences expressing the following statements:
\begin{enumerate}[(T1)]
\item  \label{axiom:o-minimal} $M\models T$,
\item  \label{axiom:c} $C\subseteq [0,1]$ is closed, has no isolated points and empty interior with $K\subseteq C$,
\item  \label{axiom:a} $A$ is an infinite, unbounded, closed and discrete subset of $M_{\geq 1}$ with initial segment $Q \subseteq A$,
\item  \label{axiom:fast} for all $n\in \N$ and all $a\in A$, if $a > \omega(n)$, then $s_{A}(a) > \exp_{n}(a)$,
\item \label{axiom:buechi} $(A,C,E,s_{A})\equiv (\N,\Cal P(\N),\in,s_{\N})$,
\item \label{axiom:buechi2} for all $c_1,\dots,c_n \in K$ and all $\Cal L_B$-formulas $\varphi$
\[
(Q,K,\epsilon,s_Q) \models \varphi(c_1,\dots,c_n) \hbox{ if and only if } (A,C,E,s_{A})\models \varphi(c_1,\dots,c_n).
\]
\item \label{axiom:ca} for all $c \in C$ and $a \in A$, $0\leq c - e(a,c) \leq a^{-1}$.
\item \label{axiom:ca2} for all $c,d\in C$ and $a\in A$,
\[
\hbox{if } 0 \leq d - e(a,c) \leq a^{-1},\hbox{ then } e(a,c)=e(a,d).
\]
\item \label{axiom:ca3} for all $c\in C$ and $a\in A$, $c - e(a,c) \in C$ and
\[
S(c-e(a,c)) = S(c)_{>a}.
\]
\item \label{axiom:succc} for all $c \in C$ and all $a,b \in A$, if $E(d,c)$ for all $d \in A \cap (a,b]$, then
\[
e(b,c) = e(a,c) + a^{-1} - b^{-1},
\]
\item \label{axiom:mu} for all $p \in \Q^n$, for all $c \in C^n$ and for all $a\in A$ with $a\geq \xi(p)$,
\begin{itemize}
\item[(i)] $p \cdot c =0$ if and only if  $\mu_p(c) = 0$.
\item[(ii)] if $0 < |p \cdot c| < a^{-1}$, then  $\mu_p(c) \geq a$.
\end{itemize}
\item \label{axiom:tau} for every $X\subseteq M^{l+n}$ and $f : X \to [0,1]$ $\Cal L$-$\emptyset$-definable continuous, and $g_1,\dots,g_k: M^{l+n}\to M$ be given as in Lemma \ref{lem:limitg}, and for every $x \in \pi(X)$ with $X_x\cap C^n\neq \emptyset$ and every $z \in \cl(f(x,X_x\cap C^n))$ there exists a lexicographically minimal $c \in C^n$ such that there is $i\in \{1,\dots,k\}$ with
$g_i(x,c) = z.$
\end{enumerate}
\end{defn}

\noindent One has to check that there is such a first-order $\Cal L_C$-theory $\TT$. In most cases it is routine to show that the above axioms can be expressed by first-order statements. Of course, statements like Axiom T\ref{axiom:fast}, Axiom T\ref{axiom:buechi} or Axiom T\ref{axiom:buechi2} have to be expressed by axiom schemes rather than by a single $\Cal L_C$-sentence.\\


\noindent Note that by Axiom T\ref{axiom:buechi} the unique $d$ in the definition of $e$ in Definition \ref{def:defofe} indeed always exists. By Axiom T\ref{axiom:buechi} we also get that the function $S$ is injective. Moreover Axiom T\ref{axiom:a} guarantees that the predecessor function $p_A$ and the successor $s_A$ on $A$ are well-defined. So in particular our use of $s_A$ in Axiom T\ref{axiom:fast} is unproblematic.

\begin{prop} $(\Cal R,K,Q,\epsilon)\models \TT$.
\end{prop}
\begin{proof} It follows immediately from the definitions of $\Cal R$, $K$ and $Q$ that Axioms T\ref{axiom:o-minimal}-T\ref{axiom:fast} hold. One can easily deduce Axioms T\ref{axiom:buechi}-T\ref{axiom:succc} from Proposition \ref{prop:isokb}. Axiom T\ref{axiom:tau} follows from Corollary \ref{cor:limitpoints}. Axiom T\ref{axiom:mu} requires a bit more explanation. We will just show the second part of Axiom T\ref{axiom:mu}, because the first statement can be shown similarly. Let $p \in \Q^n$, $c\in K^n$ and $k\in \N$ such that $0<|p \cdot c| \leq q_k^{-1}$  and $q_k \geq \xi(p)$. Towards a contradiction, suppose $\mu_p(c) < q_k$.
Note that there are $s_0,\dots s_k \in \{-1,0,1\}^n$ such that
\begin{equation}\label{eq:prooftt1}
p \cdot e(q_k,c)  = \sum_{i=0}^{k} (s_i \cdot p) q_i^{-1}.
\end{equation}
Since $\mu_p(c) < q_k$, there is at least one $i<k$ such that $s_i \cdot p\neq 0$. By the algebraic independence of elements of $Q$, we get that $\sum_{i=0}^{k-1} (s_i \cdot p) q_i^{-1}\neq 0$.
Since $q_k \geq \xi(p)$,
\begin{equation}\label{eq:prooftt2}
|\sum_{i=0}^{k-1} (s_i \cdot p) q_i^{-1}| \geq \big( (2 \max_{t \in \{-1,0,1\}^n} (t \cdot p))+1\big) \ q_k^{-1}.
\end{equation}
Note that $|p\cdot (c- e(q_k,c)) |< \max_{t \in \{-1,0,1\}^n} (t \cdot p)\cdot q_k^{-1}$ (see Axiom T\ref{axiom:ca}). From this statement, \eqref{eq:prooftt1} and \eqref{eq:prooftt2}, the reader can now easily deduce $|p\cdot c|>q_k ^{-1}$.
\end{proof}

\noindent One of the main results of this paper is the following Theorem. It will be proved towards the end of the paper.

\begin{thm}\label{thm:complete} The theory $\TT$ is complete.
\end{thm}

\subsection*{Consequences of $\TT$} In this subsection we establish first consequences of $\TT$. So throughout this subsection, let $\Cal M \models \TT$.  We start by collecting some results about the function $e$, in particular how $e$ interacts with the arithmetic operations on $M$.

\begin{lem}\label{lem:zero} The unique element $c\in C$ with $S(c)=\emptyset$ is $0$.
\end{lem}
\begin{proof} By Axiom T\ref{axiom:buechi2} $S(0)=\emptyset$. Uniqueness follows from the injectivity of $S$.
\end{proof}

\begin{lem}\label{lem:allones} Let $c\in C$ and $a \in A$ be such that $E(b,c)$ for all $b\in A$ with $b>a$. Then $c = e(a,c)+a^{-1}$.
\end{lem}
\begin{proof} Let $b \in A$ with $b>a$. By Axiom T\ref{axiom:succc}, $e(b,c) = e(a,c) + a^{-1} - b^{-1}$. Thus
\[
|c - (e(a,c)+a^{-1})| \leq  |c - e(b,c)| + |e(b,c) - (e(a,c)+a^{-1})| \leq 2b^{-1}.
\]
Since $A$ is unbounded in $M$ by Axiom T\ref{axiom:a}, $c=e(a,c)+a^{-1}$.
\end{proof}

\begin{cor}\label{cor:ainv} Let $a,b \in A$. Then $a^{-1}$ is the unique element $c$ in $C$ with $S(c)=A_{>a}$, and
\[
a^{-1} - e(b,a^{-1}) = \left\{
                \begin{array}{ll}
                  a^{-1}, & \hbox{if $b\leq a$;} \\
                  b^{-1}, & \hbox{otherwise.}
                \end{array}
              \right.
\]
\end{cor}
\begin{proof}  Let $c$ be the unique element in $C$ such that $S(c)=A_{>a}$. Note that $e(b,c)=0$ by Lemma \ref{lem:zero} for all $b\leq a$. By Lemma \ref{lem:allones}, $c = e(a,c) + a^{-1} = a^{-1}$.
Now suppose $b>a$. Then $a^{-1} - e(b,a^{-1})$ is in $C$ by Axiom T\ref{axiom:ca3} and
\[
S(a^{-1} - e(b,a^{-1})) = S(a^{-1})_{>b} = S(b^{-1}).
\]
By injectivity of $S$, we get $a^{-1} - e(b,a^{-1})=b^{-1}$.
\end{proof}

\begin{lem}\label{lem:eab} Let $c\in C$ and $a,b\in A$ with $a<b$. Then
\[
c - e(a,c) - e(b,c-e(a,c)) = c - e(b,c).
\]
\end{lem}
\begin{proof} By Axiom T\ref{axiom:ca3} $c-e(a,c) \in C$. Thus $e(b,c-e(a,c)) \in C$. Again by Axiom T\ref{axiom:ca3} $S(c - e(b,c)) = S(c)_{>b}$ and  $S(c-e(a,c)) = S(c)_{>a}$.
By Axiom T\ref{axiom:ca3} once more $c - e(a,c) - e(b,c-e(a,c))= S(c-e(a,c))_{>b} = S(c)_{>b}$. By Axiom T\ref{axiom:buechi} $c - e(a,c) - e(b,c-e(a,c)) = c - e(b,c)$.
\end{proof}

\begin{lem}\label{lem:succc} Let $c \in C$ and $a,b \in A$. If $\neg E(d,c)$ for all $d \in A \cap (a,b]$, then $e(b,c) = e(a,c)$.
\end{lem}
\begin{proof} From the assumptions on $a$ and $b$ we can directly conclude that $S(e(a,c)) = S(e(b,c))$. By Axiom T\ref{axiom:buechi} $e(a,c)=e(b,c)$.
\end{proof}

\begin{cor} \label{cor:succc} Let $c \in C$ and $a \in A$. Then
\[
e(s_{A}(a),c) = e(a,c) + \delta_{s_{A}(a),c}(a^{-1} - s_{A}(a)^{-1}).
\]
\end{cor}
\begin{proof} The statement follows immediately from Lemma \ref{lem:succc} when $\neg E(s_A(a),c)$, and from Axiom T\ref{axiom:succc} when $E(s_A(a),c)$.
\end{proof}

\noindent We get the following Lemma directly from Lemma \ref{lem:succc} and Corollary \ref{cor:succc}.

\begin{lem}\label{lem:largejump} Let $c \in C$ and let $a,b \in A$ be such that $a<b$ and $b$ is the minimal element in $A$ with $b>a$ and $E(b,c)$.
Then $e(b,c) = e(a,c) + p_A(b)^{-1} - b^{-1}.$
\end{lem}

\noindent We now collect a few easy corollaries of Axiom T\ref{axiom:mu}.

\begin{lem}\label{lem:axiommu} Let $c\in C^n$, $q\in \Q^n$ and $a \in A$ such that $a< \mu_q(c)$. Then $q\cdot e(a,c)=0$.
\end{lem}
\begin{proof} Since $a<\mu_q(c)$, $\mu_q(e(a,c))=0$. By Axiom T\ref{axiom:mu}, $q\cdot e(a,c)=0$.
\end{proof}

\begin{cor}\label{cor:mu} Let $q=(q_1,\dots, q_n) \in \Q^n$ and  $c=(c_1,\dots,c_n) \in C$ such that $\mu_q(c)>0$. Then
\[
q\cdot c =  \sum_{i=1}^{n} q_i \delta_{\mu_q(c),c_i}(p_{A}(\mu_q(c))^{-1} - \mu_q(c)^{-1}) + q \cdot(c - e(\mu_q(c),c)).
\]
\end{cor}
\begin{proof} The statement is a direct consequence of Corollary \ref{cor:succc} and Lemma \ref{lem:axiommu}.
\end{proof}



\begin{lem}\label{lem:eoflincomb} Let $a\in A$, $c=(c_1,\dots,c_n) \in C^n$, $d \in C$ and $p\in \Q^n$ such that $d = p\cdot c$. Then $e(a,d) = p \cdot e(a,c)$.
\end{lem}
\begin{proof}
 By Axiom T\ref{axiom:mu} $\mu_{(-1,p)}(d,c)=0$. By the definition of $\mu$, $\delta_{b,d} = \sum_{i=1}^n p_i \delta_{b,c_i}$ for all $b \in A$. Note that for every $c' \in C$ and $b\in A$, $\delta_{b,c'} = \delta_{b,e(a,c')}$ if $b\leq a$, and $\delta_{b,e(a,c')}=0$ otherwise. Therefore we have that $0 = - \delta_{b,e(a,d)} + \sum_{i=1}^n p_i \delta_{b,e(b,c_i)}$ for every $b\in A$. Hence $\mu_{(-1,p)}(e(a,d),e(a,c_1),\dots,e(a,c_n))=0$. By Axiom T\ref{axiom:mu} the conclusion of the Lemma follows.
\end{proof}

\noindent Because $e(b,e(a,c))=e(b,c)$ for all $c\in C$ and $a,b\in A$ with $b\leq a$, the following Corollary can be deduced directly from Lemma \ref{lem:eoflincomb}.

\begin{cor}\label{cor:eoflincomb} Let $a,b\in A$ with $b \leq a$, $c \in C^n$, $d \in C$ and $p\in \Q^n$ such that $e(a,d) = p\cdot e(a,c)$. Then
$e(b,d) = p \cdot e(b,c)$.
\end{cor}

\subsection*{Complementary intervals}
Another set of important and still rather easy consequences of $\TT$ concerns the set of complementary interval of $C$. By a \textbf{complemenatary interval} of $C$, we mean an interval $\big(x,y\big)$ between two elements $x,y \in M$ such that $\big(x,y\big)\cap C =\emptyset$ and $x,y \in C$. Given $a\in A$ and $c\in C$, we will first consider complementary intervals strictly between $e(a,c)$ and $e(a,c)+a^{-1}$. Afterwards we will consider complementary intervals that have one of these two points as an endpoint.

\begin{lem}\label{lem:complintervals} Let $a\in A$ and $c\in C$. Then
\[
\Big(e(a,c)+s_A(a)^{-1},e(a,c) + a^{-1} - s_{A}(a)^{-1}\Big)
\] is a complementary interval of $C$.
\end{lem}
\begin{proof} It follows immediately from Corollary \ref{cor:succc}, Axiom T\ref{axiom:buechi} and Lemma \ref{lem:allones} that the two endpoints of the interval are in $C$. Suppose there is $d \in C$ in the interval. By Axiom T\ref{axiom:ca2} $e(a,d)=e(a,c)$. By Corollary \ref{cor:succc} $e(s_{A}(a),d) \in \{e(a,c), e(a,c) + a^{-1} - s_{A}(a)^{-1}\}$. Hence by Axiom T\ref{axiom:ca}
either $d \in \big[e(a,c),e(a,c)+s_A(a)^{-1}\big]$ or $d\in \big[e(a,c) + a^{-1} - s_{A}(a)^{-1},e(a,c) + a^{-1}\big]$. This contradicts our assumption on $d$.
\end{proof}

\noindent We will now use Lemma \ref{lem:complintervals} to show that $e(a,c)+a^{-1}$ is the left endpoint of a complementary interval. The right endpoint of this interval will be depend on whether or not $E(a,c)$ holds.

\begin{cor}\label{cor:complintervals} Let $a,b \in A$ and $c\in C$ such that $b\in A$ is maximal in $A$ with $b\leq a$ and $\neg E(b,c)$. Then
$\Big(e(a,c)+a^{-1},e(a,c) +a^{-1} + p_A(b)^{-1} - 2b^{-1}\Big)$ is a complementary interval of $C$.
\end{cor}
\begin{proof}
Since $\neg E(b,c)$, $e(p_A(b),c)=e(b,c)$ by Lemma \ref{lem:succc}. Therefore by Lemma \ref{lem:complintervals}
$\Big(e(b,c)+b^{-1},e(b,c) + p_A(b)^{-1} - b^{-1}\Big)$ is a complementary interval of $C$. Since $e(a,c)+a^{-1} = e(b,c)+b^{-1}$ by Axiom T\ref{axiom:succc}, the statement of the Corollary follows.
\end{proof}

\noindent Note that if $a$ in the previous Corollary satisfies $\neg E(a,c)$, the right endpoint of the interval is $e(a,c)+p_A(a)^{-1}-a^{-1}$. Finally we will now show that $e(a,c)$ is the right endpoint of a complementary interval of $C$.

\begin{lem}\label{lem:complintervals2} Let $a,b\in A$ and $c\in C$ such that $b$ is maximal in $A$ with $b\leq a$ and $E(b,c)$. Then
$\Big(e(a,c)-p_A(b)^{-1}+2b^{-1},e(a,c)\Big)$ is a complementary interval of $C$.
\end{lem}
\begin{proof} By Lemma \ref{lem:complintervals} $\Big(e(p_A(b),c)+b^{-1},e(p_A(b),c) + p(b)^{-1} - b^{-1}\Big)$ is a complementary interval of $C$. By Corollary \ref{cor:succc} $e(a,c)=e(b,c) = e(p_A(b),c) + p_A(b)^{-1} - b^{-1}$. The statement of the Lemma follows.
\end{proof}

\section{The closest element in $A$ and $C$}
As before, let $\Cal M \models \TT$. In this section, we will study two important definable functions $\lambda$ and $\nu$ and their interaction with the $\Cal L$-structure on $M$.

\begin{defn} Let $\lambda : M \to A$ map $x\in M$ to $\max A \cap (-\infty,x]$ if this maximum exists, and to $0$ otherwise.
Let $\nu : M \to C$ map $x\in M$ to $\max C \cap (-\infty,x]$ if this maximum exists, and to $0$ otherwise.
\end{defn}

\noindent Since $A$ and $C$ are closed, the maximum in the definition of $\lambda$ exists for all $x \geq 1$ and the maximum in the definition of $\nu$ exists for all $x\geq 0$. The two main results of this section are as follows. The first result is that two distinct elements of $A$ have to lie in different $T$-levels. This will allows us to show that elements of $C$ that are $\Q$-linearly independent over $K$, are also $\dcl$-independent over $K$. The second main results states that if $a \in A$ and $x,y \in M$ and $x-y \in \mathfrak{m}_a$, then $e(a,\nu(y))$ can be expressed in terms of $e(a,\nu(x))$.

\subsection*{T-levels and $A$} In this subsection we will study consequences of Axiom T\ref{axiom:fast}, in particular on the $T$-levels of $M$ and on the function $\lambda$. Because $M$ is exponentially bounded, $Q\subseteq \dcl(\emptyset)$ and Axiom T\ref{axiom:fast}, we get directly the following Lemma.

\begin{lem}\label{lem:llq} Let $a \in A$. Then $0 \ll a$ iff $a \notin Q$.
\end{lem}

\noindent Lemma \ref{lem:llq} will be used routinely throughout this paper. We now show that if $a,b \in A\setminus Q$ and $a\neq b$, then $a$ and $b$ have to lie in different $T$-levels.

\begin{lem}\label{lem:oneaperlevel} Let $a,b \in A$ be such that $a\neq b$, $0 \ll a$ and $0 \ll b$. Then $[a] \cap [b] = \emptyset$.
\end{lem}
\begin{proof} With out loss of generality, assume that $a<b$. Suppose towards a contradiction that the conclusion of the Lemma fails. Then there are $\Cal L$-$\emptyset$-definable strictly increasing functions $f,g : M \to M$ such that $f(a) < b < g(a)$. Since $g$ is strictly increasing, it is invertible. Since $0\ll a$ and $0\ll b$, we get $s_A(a) > g(a)$ and $p_A(a) < f(a)$ by Axiom T\ref{axiom:fast} and exponential boundedness of $M$. Hence $p_A(a) < b < s_A(a)$. This contradicts $a\neq b$.
\end{proof}

\noindent We immediately get the following Corollary.

\begin{cor}\label{cor:oneaperlevel} Let $a_1,\dots,a_n \in A$ such that $0 \ll a_1 < a_2 < \dots < a_n$. Then
\[
[a_j] \cap \bigcup_{i\neq j} [a_i] = \emptyset.
\]
\end{cor}

\noindent Throughout the paper we will not only have to compare elements of $A$, but also their inverse. The following Lemma states that any $\Q$-linear combination of inverses of elements of $A$ is always dominated by the inverse of the smallest element.

\begin{cor}\label{cor:qlevel} Let $a_1,\dots,a_n \in A$ be such that $0 \ll a_1 < a_2 < \dots < a_n$, and let $(q_1,\dots, q_n) \in \Q^n$ with $q_1\neq 0$. Then for every $\varepsilon \in \Q_{>0}$ there are $u_1,u_2 \in \Q$ such that $u_1 a_1^{-1} < \sum_{i=1}^{n}  q_i a_i^{-1} < u_2 a_1^{-1}$ and $|q_1 -u_i|< \varepsilon$ for $i=1,2$.
\end{cor}
\begin{proof} By Corollary \ref{cor:oneaperlevel} $[a_1] < [a_2] < \dots < [a_n]$. Thus for every $r,s \in \Q_{>0}$ and $i>1$, we have $ra_1 < sa_i$ and $sa_i^{-1} < ra_1^{-1}$. The statement of the Corollary follows easily.
\end{proof}

\begin{lem}\label{lem:valsum} Let $q=(q_1,\dots, q_n) \in \Q^n$ and $c=(c_1,\dots,c_n) \in C^n$ such that $0 \ll \mu_q(c)$. Then $[|q\cdot c|^{-1}] = [p_{A}(\mu_q(c))]$.
\end{lem}
\begin{proof} 
For ease of notation, we set  $b := \mu_q(c)$. By Corollary \ref{cor:mu}
\begin{align*}
q \cdot c &= \sum_{i=1}^{n} q_i  \delta_{b,c_i} (p_{A}(b)^{-1} - b^{-1}) + q \cdot (c - e(b,c))\\
&= \sum_{i=1}^{n}  q_i  \delta_{b,c_i} p_{A}(b)^{-1} - \sum_{i=1}^{n} q_i  \delta_{b,c_i} b^{-1} +  q \cdot (c - e(b,c)).
\end{align*}
Set $u := \sum_{i=1}^{n} q_i \delta_{b,c_i}$. By Axiom T\ref{axiom:ca} we have $0 \leq c_i - e(b,c_i) \leq b^{-1}$ for each $i=1,\dots,n$. Hence there are $v_1,v_2 \in \Q$ such that $u p_{A}(b)^{-1} + v_1 b^{-1} \leq q \cdot c \leq u p_{A}(b)^{-1} + v_2 b^{-1}$. By Corollary \ref{cor:qlevel} there are $u_1,u_2 \in \Q$ such that $u_1 p_{A}(b)^{-1} \leq q\cdot c \leq u_2 p_{A}(b)^{-1}$. Thus $[|q \cdot c|^{-1}] = [p_{A}(b)].$ 
\end{proof}

\begin{cor}\label{lem:eoflincombmoda} Let $a\in A$, $c \in C^n$ and $q\in \Q^n$ such that $0\ll a$ and $q\cdot c\in \mathfrak{m}_a$. Then
$q \cdot e(a,c)=0$.
\end{cor}
\begin{proof} If $\mu_q(c) = 0$, then so is $\mu_q(e(a,c))$. By Axiom T\ref{axiom:mu} $q \cdot e(a,c)=0$. Therefore we have reduced to the case that $\mu_q(c)>0$. Since $0\ll a$,
we get that $0 < |q \cdot c| < p_A(a)^{-1}$. By Axiom T\ref{axiom:mu} $\mu_{q}(c) \geq p_A(a)$. In particular, $0 \ll \mu_q(c)$. By Lemma \ref{lem:valsum},  $[|q\cdot c|^{-1}]=[p_A(\mu_q(c))]$.
Because $q\cdot c \in \mathfrak{m_a}$, $a \leq p_A(\mu_q(c))$. By Lemma \ref{lem:axiommu} we finally get $q \cdot e(a,c)=0$.
\end{proof}

\noindent The following Corollary of Lemma \ref{lem:oneaperlevel} essentially shows that it is enough to understand $\lambda$ for a single member of each $T$-level of $M$.

\begin{cor}\label{cor:lambda} Let $x,y \in M$ such that $0 \ll x$, $0\ll y$ and $[x]=[y]$. Then
\begin{equation}\label{eq:level1}
\lambda(y) = \left\{
               \begin{array}{ll}
                 s_{A}(\lambda(x)), & \hbox{$y \geq s_{A}(\lambda(x))$;} \\
                 p_{A}(\lambda(x)), & \hbox{$y < \lambda(x)$;}\\
                \lambda(x), & \hbox{otherwise.}
               \end{array}
             \right.
\end{equation}
\end{cor}
\begin{proof} First consider the case $[x]=[\lambda(x)]$. By Lemma \ref{lem:oneaperlevel}, $\lambda(x)$ is the only element of $A$ with that property. Therefore $p_A(\lambda(x))< z < s_A(\lambda(x))$ for every $z\in [x]$.
Hence \eqref{eq:level1} holds for $y$. Now suppose $\lambda(x) < z$ for all $z \in [x]$. By Lemma \ref{lem:oneaperlevel}, $s_{A}(\lambda(x))$ is the only element of $A$ that can possibly lie in $[x]$. Again, it follows easily that \eqref{eq:level1} holds for $y$.
\end{proof}



\noindent Let $Z \subseteq C$. In the following, we say that $c\in C^n$ is \textbf{$\Q$-linearly independent over $Z$} if there are no $d\in Z^m$, $p\in \Q^m$ and non-zero $q\in \Q^n$ such that $(p,q)\cdot(d,c) = 0$.

\begin{lem}\label{lem:distinctskies} Let $c=(c_1,\dots,c_n) \in C^n$ be $\Q$-linearly independent over $K$. Then there are
\begin{itemize}
\item $d=(d_1,\dots,d_n) \in K^n$,
\item tuples $r_1,\dots,r_n,q_1,\dots,q_n \in \Q^n$ and
\item pairwise distinct $a_1,\dots,a_n \in A$
\end{itemize}
such that for every $i\leq n$,
\begin{itemize}
\item $q_{i,i} \neq 0$, $q_{i,j}=0$ if $j>i$, and
\item $0 \ll \mu_{(r_i,q_i)} (d,c) = a_i$.
\end{itemize}
\end{lem}
\begin{proof} For $i=1,\dots, n$ let $d_i$ be the unique element in $K$ such that for all $q \in Q$, we have $E(q,d_i)$ if and only if $E(q,c_i)$. The existence of such $d_i$'s follows from Axiom T\ref{axiom:buechi2}. We directly get from Axiom T\ref{axiom:mu} and $\Q$-linear independence of $c_i$ over $K$ that $0 \ll \mu_{(-1,1)}(d_i,c_i)$. We now show the statement of the Lemma by induction on $n$. \newline
For $n=1$, the conclusion of the Lemma holds with $r_{1,1}=-1$, $q_{1,1}=1$.\newline
Now suppose that the statement holds for $n-1$. Then there are pairwise distinct $a_1,\dots, a_{n-1}\in A$  and tuples
$r_1,\dots,r_{n-1},q_1,\dots,q_{n-1}\in \Q^n$ such that the conclusion of the Lemma holds for $i\leq n-1$. Without loss of generality, we can assume that $a_1 < \dots < a_{n-1}$. Let $l \in \{0,1,\dots,n-1\}$ be maximal such that there is $p=(p_1,\dots, p_{n}),s=(s_1,\dots,s_n) \in \Q^n$ with $p_{n} \neq 0$ and $\mu_{(s,p)}(d,c) = a_l$, if such an $l$ exists, and $0$ otherwise.\newline
If $l=0$, then $\mu_{(0,\dots,0,-1),(0,\dots,0,1)}(d,c) \notin \{a_1,\dots,a_{n-1}\}$. Hence the statement holds with $q_{n} =(0,\dots,0,1)$ and $r_n =(0,\dots,0,-1)$.\newline
Now consider the case that $l>0$. Set
\[
t : = \frac{\sum_{j=1}^{n} s_j \delta_{a_l,d_j} + \sum_{j=1}^{n} p_j \delta_{a_l,c_j}}{\sum_{j=1}^{n} r_{l,j} \delta_{a_l,d_j} + \sum_{j=1}^n q_{l,j} \delta_{a_l,c_j}}.
\]
The denominator is non-zero, because $\mu_{(r_l,q_l)}(d,c)=a_l$. Let $v=(v_1,\dots,v_n),w=(w_1,\dots,w_n)\in \Q^{n}$ be such that
\[
v := tr_l - s \hbox{ and } w :=tq_l - p.
\]
We will now show that $\mu_{(v,w)}(d,c)>a_l$. Since $c_1,\dots,c_n$ are $\Q$-linearly independent over $K$, $\mu_{(v,w)}(d,c) \neq 0$ by Axiom T\ref{axiom:mu}. Now let $b < a_l$.
Since $\mu_{(r_l,q_l)}(d,c)=\mu_{(s,p)}(c)=a_l$, we have
\begin{align*}
\sum_{j=1}^{n} v_j \delta_{b,d_j} + \sum_{j=1}^{n} w_j \delta_{b,c_j} &= \sum_{j=1}^{n} tr_{l,j} \delta_{b,d_j} - \sum_{j=1}^{n} s_j \delta_{b,d_j}  + \sum_{j=1}^{n} tq_{l,j} \delta_{b,c_j}- \sum_{j=1}^{n} p_j \delta_{b,c_j}\\
&= t\big(\sum_{j=1}^{n} r_{l,j} \delta_{b,d_j} + \sum_{j=1}^{n} q_{l,j} \delta_{b,c_j}\big) - \big(\sum_{j=1}^{n} s_j \delta_{b,d_j} + \sum_{j=1}^{n} p_j \delta_{b,c_j}\big)\\
& = t0 - 0 = 0.
\end{align*}
By our choice of $t$, we also get
\begin{align*}
\sum_{j=1}^{n} v_j & \delta_{a_l,d_j} + \sum_{j=1}^{n} w_j \delta_{a_l,c_j} = \sum_{j=1}^{n} tr_{l,j} \delta_{a_l,d_j}  - \sum_{j=1}^{n} s_j \delta_{a_l,d_j} + \sum_{j=1}^{n} tq_{l,j} \delta_{a_l,c_j} - \sum_{j=1}^{n} p_j \delta_{a_l,c_j}\\
&= t\big(\sum_{j=1}^{n} r_{l,j} \delta_{a_l,d_j} + \sum_{j=1}^{n} q_{l,j} \delta_{a_l,c_j}\big) - \big(\sum_{j=1}^{n} s_j \delta_{a_l,d_j} + \sum_{j=1}^{n} p_j \delta_{a_l,c_j}\big)=0.
\end{align*}
Thus $\mu_{(v,w)}(d,c)> a_l$. Since $l$ was chosen to be maximal, $\mu_{(v,w)}(c) \notin \{a_1,\dots, a_n\}$.
\end{proof}
\noindent By inspection of the proof of Lemma \ref{lem:distinctskies}, the reader can see that $d\in K^n$ is only used to make sure that $0\ll a_i$.  If this statement is dropped from the conclusion, the above proof still gives the following Lemma. Because we weaken the conclusion, we are also able to weaken the assumption of $\Q$-linear independence over $K$ to just $\Q$-linear independence.

\begin{lem}\label{lem:distinctskies2} Let $c=(c_1,\dots,c_n) \in C^n$ be $\Q$-linearly independent. Then there are
 $r_1,\dots,r_n \in \Q^n$, pairwise distinct $a_1,\dots,a_n \in A$ such that for every $i\leq n$,
\begin{itemize}
\item $q_{i,i} \neq 0$, $q_{i,j}=0$ if $j>i$, and
\item $\mu_{q_i} (c) = a_i$.
\end{itemize}
\end{lem}

\noindent Lemma \ref{lem:distinctskies} is a crucial result and will play an important role later on. The reason is that it allows us transform any given $\Q$-linearly independent tuple of elements of $C$ via linear operations into a tuple of elements of $C$ whose inverse are in different $T$-levels. Among other things the resulting elements of $C$ will not only be $\Q$-linearly independent, but also $\dcl$-independent.

\begin{prop}\label{lem:mathframma} Let $c=(c_1,\dots,c_n) \in C^n$ and let $f: M^n \to M$ be $\Cal L$-$\emptyset$-definable. Then there are $r\in \Q^n, q \in \Q^m$ and $d\in K^m$ such that for every $a\in A$ with $0\ll a$
\[
f(c) \in \mathfrak{m}_a \setminus \{0\} \Rightarrow (r,q) \cdot (d,c) \in \mathfrak{m}_a.
\]
\end{prop}
\begin{proof}
We can directly reduce to the case that $c_1,\dots, c_n$ are $\Q$-linearly independent over $K$. By Lemma \ref{lem:distinctskies} there are $d \in K^n$, $q_1,\dots,q_n,r_1,\dots,r_n \in \Q^n$ and pairwise distinct $a_1,\dots,a_n \in A$ such that for $i=1,\dots,n$, $q_{i,i}\neq 0$, $q_{i,j}=0$ for $j>i$ and $0\ll\mu_{r_i,q_i}(d,c) = a_i$.
By Lemma \ref{lem:valsum} we get that $[|(r_i,q_i) \cdot (d,c)|^{-1}] = [p_A(a_i)]$. For $i=1,\dots,n$, set $v_i :=|(r_i,q_i) \cdot (d,c)|^{-1}$. Since $a_1,\dots, a_n$ are pairwise distinct and $0\ll a_i$, we get $[v_i] \cap [v_j] = \emptyset$ for $i\neq j$ by Corollary \ref{cor:oneaperlevel}. Since $q_{i,i}\neq 0$ for each $i$ and $d\in K^n$, there is a $\Cal L$-$\emptyset$-definable function $g: M^n \to M$ such that $g(v_1,\dots, v_n) = f(c)^{-1}$.\newline
Let $a \in A$ be such that $0\ll a$ and suppose that $f(c) \in \mathfrak{m}_a$.  If $[v_i]\geq  [a]$ for some $i$, then $(r_i,q_i) \cdot (d,c) \in \mathfrak{m}_a$. Therefore the conclusion of the Proposition holds with $(r,q)=(r_i,q_i)$. We have reduced to the case that for all $i\in \{1,\dots,n\}$, $[v_i] < [a]$. Since $[v_i]\cap [v_j] = \emptyset$ for $i\neq j$, we conclude using Fact \ref{fact:tlevelfunction} that  $[g(v_1,\dots,v_n)] < [a]$. Hence $[f(c)^{-1}] < [a]$. This contradicts $f(c) \in \mathfrak{m}_a$.
\end{proof}

\begin{cor}\label{cor:mathframma}  Let $c=(c_1,\dots,c_{n-1}) \in C^{n-1}$ and let $f: M^{n-1} \to M$ be $\Cal L$-$\emptyset$-definable such that $f(c)\in C$.  Then there are $r\in \Q^m,q \in \Q^{n-1}$ and $d\in K^m$ such that
\[
f(c) = (r,q) \cdot (d,c).
\]
\end{cor}
\begin{proof} Suppose not. We can easily reduce to the case that $c_1,\dots,c_{n-1},f(c)$ are $\Q$-linearly independent over $K$.  By Lemma \ref{lem:distinctskies} there are $d \in K^n$, $q_1,\dots,q_n,r_1,\dots,r_n \in \Q^n$ and pairwise distinct $a_1,\dots,a_n \in A$ such that for $i=1,\dots,n$, $q_{i,i}\neq 0$, $q_{i,j}=0$ for $j>i$ and $0\ll\mu_{(r_i,q_i)}(d,c,f(c)) = a_i$. By Lemma \ref{lem:valsum} we get that $[|(r_i,q_i) \cdot (d,c,f(c))|^{-1}] = [p_A(a_i)]$. For $i=1,\dots,n$, set $v_i :=|(r_i,q_i) \cdot (d,c,f(c))|^{-1}$. Since $a_1,\dots, a_n$ are pairwise distinct and $0\ll a_i$, we get $[v_i] \cap [v_j] = \emptyset$ for $i\neq j$ by Corollary \ref{cor:oneaperlevel}. Since $d\in K^n$, there is a $\Cal L$-$\emptyset$-definable function $g: M^n \to M$ such that $g(v_1,\dots, v_{n-1}) = v_n$. This contradicts Fact \ref{fact:tlevelfunction}.
\end{proof}

\subsection*{Understanding $\nu$} We now turn our attention to $\nu$. By Lemma \ref{lem:complintervals}, if $x \in [0,1]\setminus C$ and $a\in A$ such that
\[
x \in \Big[e(p_A(a),\nu(x)) + a^{-1}, e(p_A(a),\nu(x)) + p_A(a)^{-1} - a^{-1}\Big),
\]
then $\nu(x) = e(p_A(a),\nu(x)) + a^{-1}$. We will show that for every $x \in [0,1]\setminus C$ there is such an $a\in A$ and that $a$ is $\Cal L_B$-definable from $\nu(x)$. We first establish the following Lemma.

\begin{lem}\label{lem:nuexplicit} Let $x \in (0,1)\setminus C$. Then there is $a\in A$ such that $\nu(x)= e(a,\nu(x)) + a^{-1}$.
\end{lem}
\begin{proof} Let $a = \lambda((x-\nu(x))^{-1})$. Because $\nu(x)\neq x$, $a\neq 0$. Indeed, $a$ is the unique element in $A$ such that $s_A(a)^{-1} <x - \nu(x) < a^{-1}$. By Axiom T\ref{axiom:ca}
\[
x < \nu(x) + a^{-1} \leq e(a,\nu(x)) + a^{-1} + a^{-1} =  e(a,\nu(x)) +2a^{-1}.
\]
First consider the case that $x < e(a,\nu(x)) + a^{-1}$. Note that
\[
e(a,\nu(x)) + s_A(a)^{-1} \leq \nu(x) + s_A(a)^{-1} < x.
\]
We will now show that $x< e(a,\nu(x)) +a^{-1} - s_A(a)^{-1}$. Suppose not. Then $x \geq e(a,\nu(x)) +a^{-1} - s_A(a)^{-1}$. Since $ e(a,\nu(x)) +a^{-1} - s_A(a)^{-1} \in C$, we get
$\nu(x) \geq e(a,\nu(x)) +a^{-1} - s_A(a)^{-1}$.
Then
\begin{align*}
x - \nu(x) &\leq  x -(e(a,\nu(x))+a^{-1} - s_A(a)^{-1})\\
&< e(a,\nu(x)) + a^{-1}-(e(a,\nu(x))+a^{-1} - s_A(a)^{-1}) = s_A(a)^{-1}.
\end{align*}
This contradicts our choice of $a$. Hence
\[
x \in \Big(e(a,\nu(x))+s_A(a^{-1}),  e(a,\nu(x)) +a^{-1} - s_A(a)^{-1}\Big).
\]
Thus $\nu(x) = e(a,\nu(x))+s_A(a^{-1})$ by Lemma \ref{lem:complintervals}. Since $e(s_A(a),e(a,\nu(x))=e(a,\nu(x))$ by the definition of $e$ and
$0< \nu(x)-e(a,\nu(x)) \leq s_A(a)^{-1}$, we get from Axiom T\ref{axiom:ca2} that $e(s_A(a),\nu(x)) = e(s_A(a),e(a,\nu(x))=e(a,\nu(x))$. So $\nu(x) = e(s_A(a),\nu(x))+s_A(a^{-1})$. Now consider that $x \geq e(a,\nu(x)) + a^{-1}$. Then $x \in \Big(e(a,\nu(x)) + a^{-1},e(a,\nu(x)) +2a^{-1}\Big)$. By Corollary \ref{cor:complintervals} $\nu(x) = e(a,\nu(x)) + a^{-1}$.
\end{proof}

\noindent We immediately get the following Corollary. This Corollary implies that the interpretation of $\nu$ in $(\Cal R,K)$ is Borel.

\begin{cor}\label{cor:nuboreldef} Let $x \in [0,1]$ and $c \in C$. Then $\nu(x)=c$ iff either
\begin{itemize}
\item $x\in C$ and $c=x$, or
\item there exists $a \in A$ and $d \in C$ with $d=e(a,d)$ and $c=d+a^{-1}$.
\end{itemize}
\end{cor}

\begin{cor}\label{cor:nuexplicit} Let $x \in [0,1]\setminus C$ and $d\in A$ such that $d$ is maximal in $A$ with $\neg E(d,\nu(x))$. Then $\nu(x) = e(p_A(d),\nu(x)) + d^{-1}$.
\end{cor}
\begin{proof} By Lemma \ref{lem:nuexplicit} there is $a\in A$ such that $\nu(x) = e(a,\nu(x)) + a^{-1}$. By Lemma \ref{lem:allones} we get that $E(b,\nu(x))$ for all $b\in A$ with $b>a$. Hence there is a maximal $d \in A$ such that $\neg E(d,\nu(x))$. Then by Lemma \ref{lem:allones}
$\nu(x) = e(p_A(d),\nu(x)) + d^{-1}$. 
\end{proof}

\noindent We now establish that we can express the image of a $\Q$-linear combination of elements of $C$ under $\nu$ as a $\Q$-linear combination of images of the elements under $e$.

\begin{lem}\label{lem:nusum} Let $c=(c_1,\dots, c_n) \in C^n$, $q=(q_1,\dots, q_n) \in \Q^n$ with $0 < q \cdot c < 1$ and let $a$ be the minimal element in $A$ such that $\sum_{i=1}^{n} q_i \delta_{s(a),c_i} \notin \{0,1\}$. If $0 \ll a$,  then
\[
\nu(q \cdot c)= \left\{
                              \begin{array}{ll}
                                s(a)^{-1} + q \cdot e(a,c), & \hbox{if $0<\sum_{i=1}^n q_i \delta_{s(a),c_i}<1$;} \\
                                a^{-1} + q \cdot e(a,c), & \hbox{if $1<\sum_{i=1}^n q_i \delta_{s(a),c_i}$;}\\
                             b^{-1} + q \cdot e(b,c), & \hbox{otherwise,}
                             \end{array}
                            \right.
\]
where $b$ is the largest element in $A$ with $b \leq a$ such that $\sum_{i=1}^{n} q_i \delta_{b,c_i}=1$.
\end{lem}
\begin{proof} Let $d \in C$ be such that for all $a' \in A$, the following statement is true: if $a'\leq a$, then $E(a',d) \hbox{ iff } \sum_{i=1}^{n} q_i \delta_{a,c_i} = 1$, and if $a'>a$, then $\neg E(a',d)$. By Axiom T\ref{axiom:buechi} such a $d$ exists. Hence $\delta_{a',d} = \sum_{i=1}^{n} q_i \delta_{a',c_i}$ for all $a' \leq a$, and thus $\mu_{(q,-1)}(c,d)=s_A(a)$ by our choice of $a$. Since $e(a,d)=d$, we get that
$d=q \cdot e(a,c)$ by Axiom T\ref{axiom:mu}. By Corollary \ref{cor:mu}
\[
q \cdot c -d =  \sum_{i=1}^n q_i \delta_{s_A(a),c_i} (a^{-1} -s_A(a)^{-1}) + q \cdot (c- e(s_A(a),c)).
\]
Set $u:=  \sum_{i=1}^n q_i \delta_{s_A(a),c_i}$. By Axiom T\ref{axiom:ca}, $c_i -e(s_A(a),c_i) < s_A(a)^{-1}$ for each $i=1,\dots,n$.
First consider the case that $0<u<1$. By Corollary \ref{cor:qlevel} there are $u_1,u_2 \in \Q$ with $0<u_1<u_2<1$ such that $u_1 a^{-1} < q\cdot c -d < u_2 a^{-1}$. Since $s_A(a)>ra$ for every $r\in \Q_{>0}$, we get
\begin{align*}
d &< d + s_A(a)^{-1} < d + u_1 a^{-1} < q\cdot c < d + u_2 a^{-1} < d + a^{-1} - s_A(a)^{-1}.
\end{align*}
By Lemma \ref{lem:complintervals} we get $\nu(q \cdot c) = d + s_A(a)^{-1}$. Now suppose that $u > 1$. By Corollary \ref{cor:qlevel}
there are $u_1,u_2 \in \Q$ with $1<u_1<u_2$ such that $u_1 a^{-1} <q \cdot c -d < u_2 a^{-1}$. Suppose $E(a',d)$ holds for all $a' \leq a$. Then $d = 1 - a^{-1}$ and $q\cdot c >1$. Hence we can assume there is
$a'\in A$ maximal such that $a'\leq  a$ and $\neg E(a',d)$. Then $e(a,d)+a^{-1} = e(a',d) + a'^{-1}$ by Axiom T\ref{axiom:succc}. Since $\neg E(a',d)$, $e(p_A(a'),d)=e(a',d)$. Thus
\begin{align*}
e(p_A(a'),d) + a'^{-1} &= d + a^{-1} < d + u_1 a^{-1} < q \cdot c\\
&< d + u_2 a^{-1} < e(p_A(a'),d) + p_A(a')^{-1} - a'^{-1}.
\end{align*}
Hence by Lemma \ref{lem:complintervals} $\nu(q\cdot c) = d + a^{-1}$. Suppose that $u < 0$. There are $u_1,u_2 \in \Q$ with $u_1<u_2<0$ such that $u_1 a^{-1} <q\cdot c -d < u_2 a^{-1}$.
Let $b\in A$ be the largest element in $A$ with $b \leq a$ such that $E(b,d)$. Because $q\cdot c>0$, such a $b$ exists. By Lemma \ref{lem:succc} and  Corollary \ref{cor:succc} $d = e(b,d) = e(p_A(b),d) + p_A(b)^{-1} - b^{-1}$.
By Axiom T\ref{axiom:fast} $b^{-1} < p_A(b)^{-1} - b^{-1} +u_1 a^{-1}$.
Therefore
\begin{align*}
e(p_A(b),d) + b^{-1}& <  e(p_A(b),d) + p_A(b)^{-1} - b^{-1} +u_1 a^{-1} \\
&< d + u_1a^{-1}< q \cdot c < d+ u_2 a^{-1}\\
&< d =e(p_A(b),d) + p_A(b)^{-1} - b^{-1}.
\end{align*}
Hence $\nu(q \cdot c) = d + b^{-1}$ by Lemma \ref{lem:complintervals}.
\end{proof}

\subsection*{Closed elements in images under $\Cal L$-definable functions} Before finishing this section, we need to mention two further classes of $\Cal L_C$-definable functions. While $\nu$ maps an element $z$ of $(0,1)$ to the left endpoint of the complementary interval of $C$ in whose closure $z$ lies, we also have to understand functions that map $z$ to left and right endpoints of complementary intervals of the closure of the image of $C^n$ under a $\Cal L$-definable function. Here Axiom T\ref{axiom:tau} is the key.

\begin{defn} \label{def:nutau}
Let $f : X\subseteq M^{l+n} \to [0,1]$ be $\Cal L$-$\emptyset$-definable and continuous. Let $g_1,\dots,g_k : M^{l+n}\to M$ be as in Axiom T\ref{axiom:tau}.
Define $\nu_f : M^{l+1} \to C^n$ to map $(x,y) \in M^{l+1}$ to the lexicographically minimal $c\in C^n\cap \cl(X_x)$ such that there is $i\in \{1,\dots,k\}$ with
\[
g_i(x,c) = \sup_{d \in C^n \cap X_x, f(x,d)\leq y} f(x,d),
\]
when $C^n \cap X_x\neq \emptyset$, and to $0$ otherwise. Let $\tau_f : M^{l+1} \to C^n$ map $(x,y) \in M^{l+1}$ to the lexicographically minimal $c\in C^n\cap \cl(X_x)$ such that there is $i\in \{1,\dots,k\}$ with
\[
g_i(x,c) = \inf_{d \in C^n \cap X_x, f(x,d)\geq y} f(x,d),
\]
when $C^n \cap X_x\neq \emptyset$, and to $0$ otherwise.
\end{defn}

\noindent The existence of the lexicographically minimal elements of $C^n$ in Definition \ref{def:nutau} follows immediately from Axiom T\ref{axiom:tau} when $x\in \pi(X)$ and $y$ is bounded above and below by an element of $f(x,C^n \cap X_x)$. One can deduce the following Lemma easily from Definition \ref{def:nutau}.

\begin{lem} Let $f : X\subseteq M^{l+n} \to [0,1]$ be $\Cal L$-$\emptyset$-definable and continuous, and let $x \in \pi(X)$ and $y \in M$. Let $g_1,\dots,g_k : M^{l+n}\to M$ be as in Axiom T\ref{axiom:tau}. If there are $c,d \in C^n \cap X_x$ such that $f(x,c) < y < f(x,d)$, then there are $i,j \in \{1,\dots,k\}$
\[
\Big(g_i(x,\nu_f(x,y)),g_j(x,\tau_f(x,y))\Big) \cap f(C^n \cap X_x) = \emptyset.
\]
\end{lem}

\noindent Consequently for every $y$ in the convex closure of $f(x,C^n \cap X_x)$ but not in $f(x,C^n \cap X_x)$, there are $i,j$ such that $g_i(x,\nu_f(x,y))$ and $g_j(x,\tau_f(x,y))$ are the endpoints of the complementary interval of $f(x,C^n \cap X_x)$ whose closure contains $y$. Thus $\nu_f(x,-)$ and $\tau_f(x,-)$ are constant on complementary intervals of $f(x,C^n \cap X_x)$.


\begin{cor} Let $f : X\subseteq M^{l+n} \to [0,1]$ be $\Cal L$-$\emptyset$-definable and continuous, $x\in \pi(X)$ and  let $g_1,\dots,g_k : M^{l+n}\to M$ be as in Axiom T\ref{axiom:tau}. Then
\begin{itemize}
\item [(i)] $\nu_f(x,-)$ and $\tau_f(x,-)$ are continuous on $M \setminus \cl(f(x,X_x\cap K^n))$,
\item [(ii)] if $y \in \cl(f(x,X_x\cap K^n))$, then $y = g_i(x,\nu_f(x,y))$ or $y = g_i(x,\tau_f(x,y))$ for some $i$.
\end{itemize}
\end{cor}

\section{B\"uchi expressibility}
Recall that the language of $(\N,\Cal P(\N),\in,s_{\N})$ was denote by $\Cal L_B$. Let $\Cal M\models \TT$.
We will write $\Cal B(M)$ for the $\Cal L_B$-structure $(A,C,E,s_{A})$. In this section, we will study this structure in detail. In particular, we will show that certain algebraic conditions on elements of $C$ are equivalent to statements expressible in $\Cal B(M)$.

\begin{lem}\label{lem:definv} Let $a \in A$ and $c \in C$. Then
\begin{itemize}
\item[(i)] $a^{-1}$ is $\Cal L_B$-definable from $a$,
\item[(ii)] $c-e(a,c)$ is $\Cal L_B$-definable from $a$ and $c$.
\end{itemize}
\end{lem}
\begin{proof} For (i), by Corollary \ref{cor:ainv} $a^{-1}$ is the unique element $d$ in $C$ such that $E(b,d)$ holds iff $b> a$ for all $b$. Since this property is $\Cal L_B$-definable over $a$, $a^{-1}$ is $\Cal L_B$-definable from $a$. For (ii), note that by Axiom T\ref{axiom:ca2} $c-e(a,c)$ is the unique element $d \in C$ with $S(d)=S(c)_{>a}$. Because this property is $\Cal L_B$-definable over $a$ and $c$, $c-e(a,c)$ is $\Cal L_B$-definable from $a$ and $c$.
\end{proof}

\begin{prop}\label{prop:buechiexp} Let $q=(q_1,\dots, q_n)\in \Q^n$. There are $\Cal L_B$-formulas $\chi_1(x,y)$, $\chi_2(x,y)$, $\varphi(x,y)$, $\psi(x,y)$, $\theta(x)$, $\omega(x)$ such that for every $a \in A$ and every $c=(c_1,\dots,c_n)\in C^n$
\begin{itemize}
\item[(i)] $\sum_{i=1}^{n} q_i \delta_{a,c_i} = 0$ iff $\Cal B(M) \models \chi_1(c,a)$.
\item[(ii)] $\sum_{i=1}^{n} q_i \delta_{a,c_i} > 0$ iff $\Cal B(M) \models \chi_2(c,a)$.
\item[(iii)] $\mu_q(c) \geq a$ iff $\Cal B(M) \models \varphi(c,a)$,
\item[(iv)] $\mu_q(c) = a$ iff $\Cal B(M) \models \psi(c,a)$,
\item[(v)] $q \cdot c = 0$ iff $\Cal B(M) \models \theta(c)$,
\item[(vi)] $q \cdot c \in C$ iff $\Cal B(M) \models \omega(c)$,
\end{itemize}
\end{prop}
\begin{proof} It is easy to see that there are $\Cal L_B$-formulas $\chi_1(x,y),\chi_2(x,y)$ that satisfy (i) and (ii). For (iii), let $\varphi(x,y)$ be the formula $\forall z \in A \ (z < y) \rightarrow \chi_1(x,z)$. It follows immediately from the definition of $\mu$ that (iii) holds with this choice of $\varphi$. For (iv), let $\psi(x,y)$ be $\neg \chi_1(x,y) \wedge \varphi(x,y)$. Again it follows immediately from the definition of $\mu$ that (iv) holds for this $\psi$. For (v), let $\theta(x)$ be the formula $\forall y \in A \ \chi_1(x,y)$. We now show that (v) holds. First, suppose that $\Cal B(M) \models \theta(c)$. By the definitions of $\mu$ and $\chi_1$ we have that $\mu_q(c) =0$. By Axiom T\ref{axiom:mu} $q \cdot c = 0$. Now suppose that $q \cdot c = 0$. Then by Axiom T\ref{axiom:mu} $\mu_q(c)=0$.
Thus $\Cal B(M) \models \theta(c)$. For (vi), note that by (v) there is a $\Cal L_B$-formula $\theta'(x,y)$ such that $\Cal B(M) \models \theta'(c,d)$ iff $d=q \cdot c$. Now set $\omega(x)$ to be the $\Cal L_B$-formula $\exists y \in C \ \theta'(x,y)$. Statement (vi) follows.
\end{proof}

\noindent The $\Cal L_B$-formulas in Proposition \ref{prop:buechiexp} dependent on the given tuple $q$.


\section{Towards quantifier elimination}

In this section the first steps toward a quantifier elimination statement for $\TT$ are made. We will not show that the theory $\TT$ has quantifier elimination in the language $\Cal L_C$. Rather we extend this language and theory by definitions to a language $\Cal L_C^+$ and a theory $\TT^+$ and show quantifier elimination for this expansion. \newline

\noindent For this section, we fix a model $\Cal M=(M,C,A,E)\models \TT$. In the following we will consider substructures of $M$. Whenever $X$ is a $\Cal L_C$-substructure of $\Cal M$, we denote by $C(X)$ and $A(X)$ the interpretation of the symbols for $C$ and $A$ in $X$. Since $X$ is a substructure, $C(X) = X\cap C$ and $A(X) = X \cap A$. Whenever $A(X)$ is closed under $s_A$, we will write $\Cal B(X)$ for the two-sorted structure $(C(X),A(X),E|_{A(X)\times C(X)},s_{A(X)})$. We write that $a\in A$ is $\Cal L_B$-definable from $X$ whenever $a$ is $\Cal L_B$-definable from $A(X)\cup C(X)$ in $\Cal B(M)$.

\subsection*{Languages and theories} We will now introduce the expansions of $\Cal L_C$ and $\TT$. Since $(A,C,E,s_{A})$ is definable in $\Cal M$, it is easy to see that for every $\Cal L_B$-formula $\varphi(x_1,\dots,x_m,y_1,\dots,y_n)$, where the variables $x_1,\dots,x_m$ are of the first sort and the variables $y_1,\dots, y_n$ are of the second sort, there is a $\Cal L_C$-formula $\varphi^C$ such that
\[
\Cal M \models \varphi^C(a,c) \hbox{ iff } a \in A^m, c\in C^n \hbox{ and } \Cal B(M) \models \varphi(a,c).
\]
Whenever $X,Y \subseteq M$ , we simply write $\tp_{\Cal L_B}(Y | X)$ for $\tp_{\Cal L_B}(\Cal B(Y) | \Cal B(X))$.\newline

\noindent Let $\Cal L_C^B$ be the language $\Cal L_C$ augmented by $n$-ary predicate symbols $P_{\varphi}$ for each $\Cal L_B$-formula $\varphi$ in $n$ free variables. Let $\TT^B$ be the $\Cal L_C^B$-theory extending $\TT$ by axioms
\[
\forall x  \ (P_{\varphi}(x) \leftrightarrow \varphi^C(x))
\]
for every $\Cal L_B$-formula $\varphi$. \newline

\noindent Let $\Cal L_C^*$ be the language $\Cal L_C^B$ augmented by function symbols for $\lambda$, $\nu$, $e$ and for every $\Cal L_B$-$\emptyset$-definable $g: A^m \times C^n \to A$. Let $\Cal L_C^{+}$ be the language $\Cal L_C^*$ augmented by function symbols for $\nu_f$ and $\tau_f$ for every $\Cal L$-$\emptyset$-definable function $f$ that satisfies the assumption of Definition \ref{def:nutau}. For each of the functions symbols $f$ added let $\varphi_f(x,y)$ be the $\Cal L_C$-formula defining the $\Cal L_C$-definable function corresponding to $f$. Let $\TT^+$ be the $\Cal L_C^*$-theory extending $\TT^B$ by axioms
\[
\forall x  \ (\varphi_f(x,y) \leftrightarrow f(x)=y)
\]
for each new function symbol in $\Cal L_C^+$.\newline

\noindent Note that every model of $\TT$ naturally extends to a model of $\TT^+$. From now on we will regard each $\TT$-model also as a $\TT^+$-model whenever needed. If $X\subseteq M$ is a $\Cal L_C^*$-substructure of $M$, we write $X \unlhd^* M$. If $X\subseteq M$ is even a $\Cal L_C^+$-substructure of $M$, we write $X \unlhd^+ M$. Our main quantifier elimination results can now be stated as follows.

\begin{thm}\label{thm:qe} $\TT^+$ has quantifier elimination.
\end{thm}

\subsection*{$\Cal L_C^+$-substructures} Towards proving Theorem \ref{thm:qe} substructures of $\Cal M$ in the extended languages $\Cal L_C^*$ and $\Cal L_C^+$ will be studied in this subsection. We are in particular interested in the question how to extend $\Cal L_C^+$-substructures to a $\Cal L_C^*$-substructure that contains a given subset of $C$. Before making this statement precise, we will prove several Lemmas. We remind the reader that $\dcl$ denotes the definable closure operator in the o-minimal reduct $M$.

%

\begin{lem}\label{lem:basis} Let $X \unlhd^+ \Cal M$. Then there is $Z\subseteq X$ such that $Z$ is $\dcl$-independent over $C$ and $X=\dcl(Z\cup C(X))$.
\end{lem}
\begin{proof} Let $Z\subseteq X$ be maximal such that $Z$ is $\dcl$-independent over $C$. It is left to show that $X=\dcl(Z\cup C(X))$. Let $x \in X$. Without loss generality, we can assume that $x \in [0,1]$.
 By maximality of $Z$, there is $z \in Z^m$ and $c\in C^n$ such that $x\in \dcl(z,c)$. By o-minimality of $T$ we can assume that there is an $\emptyset$-definable open cell $U$, an $\Cal L$-$\emptyset$-definable continuous function $f: U \to [0,1]$ such that $(z,c) \in U$ and $f(z,c)=x$. Let $g_1,\dots, g_k$ be as in Definition \ref{def:nutau}. Since $f(z,c)=x$, there is $i\in \{1,\dots,k\}$ such that $g_i(z,\tau_f(z,x))=x$. Because $x,z \in X$ and $X \unlhd^+ \Cal M$, $\tau_f(z,x)\in C(X)^n$. Therefore $x \in \dcl(Z\cup C(X))$.
\end{proof}

\begin{lem}\label{lem:corv} Let $X\unlhd^+M$, $z\in (X\cap [0,1])^m$, $y \in X$ and $a\in A$ such that $0\ll a$, $\overline{y}^a$ is $\dcl$-dependent over $\overline{z}^a$ and $\overline{C}^{a}$. Then $\overline{y}^a$ is $\dcl$-dependent over $\overline{z}^a$ and $\overline{C(X)}^{a}$.
\end{lem}
\begin{proof} We can assume that $y \in [0,1]$. By Lemma \ref{lem:overlinedep} there is an open $\Cal L$-$\emptyset$-cell $U$, a continuous $\Cal L$-$\emptyset$-definable function $f : U\to M$ and $c\in C^n$ such that
 $[|f(z,c)-y|^{-1}]\geq [a]$. Let $g_1,\dots,g_k$ be given as in Definition \ref{def:nutau}. Then there is $i\in \{1,\dots,k\}$ such that $[|g_i(z,\nu_f(z,y))-y|^{-1}]\geq [a]$ or $[|g_i(z,\tau_f(z,y))-y|^{-1}]\geq [a]$.
Since $X\unlhd^+ M$, $\nu_g(x,y),\tau_g(x,y) \in C(X)^n$. It can easily be deduced from Lemma \ref{lem:overlinedep} that $\overline{y}^a$ is $\dcl$-dependent over $\overline{z}^a$ and $\overline{C(X)}^{a}$.
\end{proof}

\begin{lem}\label{lem:corv2} Let $X \unlhd^+ \Cal M$, let $c \in C^n$, $a\in A$ such that $0\ll a$ and $\overline{c}^a$ is $\dcl$-dependent over $\overline{X}^a$. Then $\overline{c}^a$ is $\dcl$-dependent over $\overline{C(X)}^a$.
\end{lem}
\begin{proof} Take $z \in (X\cap [0,1])^l$, $d \in C(X)^m$ such that $\overline{c}^a$ is $\dcl$-dependent over $\overline{(z,d)}^{a}$ and $\overline{z}^a$ is $\dcl$-independent over $\overline{C(X)}^{a}$.  By Lemma \ref{lem:corv} $\overline{z}^{a}$ is $\dcl$-independent over $\overline{C}^{a}$. Since $(d,c) \in C^{m+n}$ and $\overline{c}^a$ is $\dcl$-dependent over $\overline{(z,d)}^{a}$, $\overline{c}^a$ is $\dcl$-dependent over $\overline{d}^{a}$.
\end{proof}


\begin{lem}\label{lem:eavf0} Let $X\unlhd^+ \Cal M$, $c \in C^n$ and $f: M^n \to M$ be $\Cal L$-$X$-definable such that $f(c)\in C$. Then there are
$x \in C(X)^m$ and $p\in \Q^m, q \in \Q^n$ such that
\[
f(c) = (p,q)\cdot (x,c).
\]
\end{lem}
\begin{proof} Because $(c,f(c))\in C^{n+1}$, it follows from Lemma \ref{lem:basis} that there are $x\in C(X)^m$ and a $\Cal L$-$\emptyset$-definable function $g: M^{m+n} \to M$ such that $g(x,c) =f(c)$.
By Corollary \ref{cor:mathframma} we can extend $x$ by elements of $K$ such that there are $p\in \Q^m, q \in \Q^n$ with $f(c) = (p,q)\cdot (x,c)$.
\end{proof}

\noindent Hence any $\Cal L$-dependence among elements of $C$ over a $\Cal L_C^+$-substructure is just a $\Q$-linear dependence. We immediately get the following Corollary.

\begin{cor}\label{cor:eavf0} Let $X\unlhd^+\Cal M$, $D\subseteq C$ and $Y:=\dcl(X\cup D)$. Then every element in $C(Y)$ is a $\Q$-linear combination of elements in $C(X) \cup D$.
\end{cor}

\begin{lem}\label{lem:eavf1} Let $X\unlhd^+\Cal M$, let $c \in C^n, d\in C, a\in A$ with $0\ll a$ and let $f: M^n \to M$ be $\Cal L$-$X$-definable such that $f(c)-d \in \mathfrak{m}_a$. Then there are
$x \in C(X)^m$ and $p\in \Q^m, q \in \Q^n$ such that $e(a,d) = (p,q)\cdot (e(a,x),e(a,c))$.
\end{lem}
\begin{proof} By Lemma \ref{lem:corv2} and Lemma \ref{lem:overlinedep} there are $x\in C(X)^m$ and $\Cal L$-$\emptyset$-definable function $g: M^{m+n+1} \to M$ such that $g(x,c,d) \in \mathfrak{m}_a$. Then by Proposition \ref{lem:mathframma} we can extend $x$ by elements from $K$ such that there are $p\in \Q^m, q \in \Q^n$ with $d - (p \cdot x + q \cdot c) \in \mathfrak{m}_a$. By Corollary \ref{lem:eoflincombmoda}
$e(a,d) = (p,q)\cdot (e(a,x),e(a,c))$.
\end{proof}

\subsection*{Extensions by elements of $C$} In the following we will need to consider a special kind of $\Cal L_C^*$-substructures of $\Cal M$. These substructures are given by an extension of a $\Cal L_C^+$-substructure by elements of $C$.

\begin{defn} A subset $X \subseteq M$ is a \textbf{special $\Cal L_C^*$-substructure} if there are $D\subseteq C$ and $Z\subseteq M$ such that $\dcl(Z \cup D) \unlhd^* M$  and $Z$ is either $\emptyset$ or $Z \unlhd^+ M$. In this case, we write $X \unlhd M$.
\end{defn}
\noindent While Lemma \ref{lem:eavf0} and Lemma \ref{lem:eavf1} do not generalize to arbitrary $\Cal L_C^*$-substructures, both statements hold for special substructures as can easily be checked.

\begin{lem}\label{lem:eavf02} Let $X\unlhd \Cal M$, $c \in C^n$ and $f: M^n \to M$ be $\Cal L$-$X$-definable such that $f(c)\in C$. Then there are
$x \in C(X)^m$ and $p\in \Q^m, q \in \Q^n$ such that
\[
f(c) = (p,q)\cdot (x,c).
\]
\end{lem}

\begin{lem}\label{lem:eavf12} Let $X\unlhd \Cal M$, let $c \in C^n, d\in C, a\in A$ with $0\ll a$ and let $f: M^n \to M$ be $\Cal L$-$X$-definable such that $f(c)-d \in \mathfrak{m}_a$. Then there are
$x \in C(X)^m$ and $p\in \Q^m, q \in \Q^n$ such that $e(a,d) = (p,q)\cdot (e(a,x),e(a,c))$.
\end{lem}

\begin{defn} A set $Z\subseteq M$ is \textbf{$A$-closed} if $Z$ is closed under $\lambda$ and contains all $a \in A$ that are $\Cal L_B$-definable from $Z$.
\end{defn}

\noindent We are now ready to state the main result of this subsection.

\begin{prop}\label{cor:aclosure} Let $X \unlhd\Cal M$, $Y\subseteq \Cal M$ and $D\subseteq C$  such that
\begin{itemize}
\item [(i)] $Y = \dcl(X\cup D)$,
\item [(ii)] $D$ is closed under $e(a,-)$ for each $a \in A(Y)$, and
\item [(iii)] $Y$ is $A$-closed.
\end{itemize}
Then $Y\unlhd \Cal M$.
\end{prop}
\begin{proof}  Because $D\subseteq C$, it is only left to show that $Y$ is closed under $\nu$ and $e(a,-)$ for every $a \in A(Y)$. We start by showing the latter statement. Let $y \in C(Y)$ and $a \in A(Y)$. By (i) there is a $\Cal L$-$X$-definable function $f: M^n \to M$ such that $y=f(c)$ for some $c=(c_1,\dots,c_n) \in D^n$. By Lemma \ref{lem:eavf02} there are  $x\in C(X)^m$, $p\in \Q^m$ and $q \in \Q^n$ such that $y = (p,q)\cdot(x,c)$. By Lemma \ref{lem:eoflincomb} $e(a,y) = (p,q) \cdot (e(a,x),e(a,c))$. By (ii) $e(a,c) \in D^n$. By (i) $e(a,y)\in Y$.\newline

\noindent We will now check that $Y$ is closed under $\nu$. Let $y\in Y \cap (0,1)$. We immediately reduce to the case that $\nu(y)\neq y$. By (i) there is an $\Cal L$-$X$-definable function $f: M^n \to M$ such that $y=f(c)$ for some $c\in D^n$. By Corollary \ref{cor:nuexplicit} there is $b\in A$ maximal such that $\neg E(b,\nu(f(c)))$ and $\nu(f(c))=e(p_A(b),\nu(f(c)))+b^{-1}$. We can assume that $0 \ll b$. Hence
\begin{align*}
e(p_A(b),\nu(f(c)) &< e(p_A(b),\nu(f(c))) + b^{-1} = \nu(f(c))\\
 &< f(c) < e(p_A(b),\nu(f(c))) + p_A(b)^{-1} -  b^{-1} \\
 &< e(p_A(b),\nu(f(c))) + p_A(b)^{-1}.
\end{align*}
We first observe that $\nu(f(c)) + p_A(b)^{-1} -2b^{-1} = e(p_A(b),\nu(f(c))) + p_A(b)^{-1} -  b^{-1}$. Set
\begin{align*}
a_1 &:= \lambda( (f(c) - \nu(f(c))^{-1} ),\\
a_2 &:= \lambda( (e(p_A(b),\nu(f(c))) + p_A(b)^{-1} -  b^{-1} - f(c))^{-1} ).
\end{align*}
Note that $a_1,a_2 \geq p_A(b)$. Set $a:= \max \{a_1,a_2\}$. Thus either
\[
\nu(f(c)) - f(c) \in \mathfrak{m}_a \hbox{ or } \nu(f(c)) + p_A(b)^{-1} -  2b^{-1} - f(c) \in \mathfrak{m}_a.
\]
By Lemma \ref{lem:eavf12} there are $x\in C(X)^m$, $p\in \Q^m$ and $q \in \Q^n$ such that either
\[
\nu(f(c)) - (p,q)\cdot (x,c) \in \mathfrak{m}_a \hbox{ or } \nu(f(c)) + p_A(b)^{-1} -  2b^{-1} - (p,q)\cdot (x,c) \in \mathfrak{m}_a
\]
Set $z:=(p,q)(x,c)$. By Lemma \ref{lem:nusum} and assumption (iii), we have that $\nu(z) \in Y$. We will now consider several different cases. In each case we will conclude that $\nu(f(c)) \in Y$. The arguments that follow are not complicated, but the details are tiresome. Because the arguments itself are not crucial for the rest of the paper, the reader might prefer to skip them or just read the the first two case which already contain the main ideas.\newline

\noindent First, consider the case that $e(p_A(b),\nu(f(c))) + b^{-1} \leq z  < e(p_A(b),\nu(f(c))) + p_A(b)^{-1} - b^{-1}$. By Lemma \ref{lem:complintervals}
$\nu(z) = e(p_A(b),\nu(f(c)) + b^{-1}$. Hence $\nu(f(c))=\nu(z)\in Y$.\newline

\noindent Now consider that $z < e(p_A(b),\nu(f(c)))$. Let $d \in A$ be maximal such that $d \leq p_A(b)$ and $E(d,\nu(f(c)))$. By Lemma \ref{lem:complintervals2}
\begin{equation}\label{eq:yspecialeq1}
C \cap \big( e(p_A(b),\nu(f(c))) - p_A(d)^{-1} + 2d^{-1},e(p_A(b),\nu(f(c))\big) = \emptyset.
\end{equation}
Because $0< e(p_A(b),\nu(f(c)))-z < \nu(f(c)) - z \in \mathfrak{m}_a$ and $a \geq p_A(b) > p_A(d)$, we get that $\nu(z) = e(p_A(b),\nu(f(c)))-p_A(d)^{-1}+2d^{-1}.$ Since $z<e(p_A(b),\nu(f(c)))$, we have $\nu(f(c))-z>b^{-1}$. Since $\nu(f(c))-z \in \mathfrak{m_a}$, $f(c)-\nu(f(c)) < p_A(b)^{-1} + 2b^{-1}$ and $a \geq p_A(b)$, we can deduce that $b^{-1} <f(c) - z < (p_A(p_A(b)))^{-1}.$ Hence $\lambda((f(c)-z)^{-1})$ has to be either $p_A(b)$ or $p_A(p_A(b))$. Therefore $b \in Y$, since $\lambda((f(c)-z)^{-1})\in Y$ and $Y$ is closed under $p_A$ and $s_A$. Let $d' \in A$ be the largest element in $A$ such that $\neg E(d',\nu(z))$. Because $\nu(z) \in Y$ and $Y$ is $A$-closed, $d' \in Y$. By Corollary \ref{cor:nuexplicit}, $\nu(z) = e(p_A(d'),\nu(z))+d'^{-1}$ and $e(p_A(b),\nu(f(c))) = e(p_A(d'),\nu(z)) + p_A(d')^{-1} -d'^{-1}$ by \eqref{eq:yspecialeq1}. Thus $e(p_A(b),\nu(f(c)))\in Y$. Since $b \in Y$ and  $\nu(f(c))=e(p_A(b),\nu(f(c)))+b^{-1}$, we get that $\nu(f(c)) \in Y$.\newline

\noindent Consider the case that $z > e(p_A(b),\nu(f(c))) + p_A(b)^{-1}$. Note that
\[
z-(e(p_A(b),\nu(f(c)))+p_A(b)^{-1}) <z - (\nu(f(c)) + p_A(b)^{-1} -  2b^{-1}) \in \mathfrak{m}_a.
\]
Since $a \geq p_A(b)$, we get from Corollary \ref{cor:complintervals} that $\nu(z)=e(p_A(b),\nu(f(c)))+p_A(b)^{-1}$. Because $f(c) - (\nu(f(c)) + p_A(b)^{-1} -  2b^{-1}) \in \mathfrak{m}_a$, $z- (\nu(f(c)) + p_A(b)^{-1} -  2b^{-1}) \in \mathfrak{m}_a$ and $\nu(z)- (\nu(f(c)) + p_A(b)^{-1} -  2b^{-1})=b^{-1}$, we can again deduce that $b^{-1} \leq z- f(c) < p_A(p_A(b))^{-1}$. Hence $\lambda((z-f(c))^{-1})$ has to be either $p_A(b)$ or $p_A(p_A(b))$. Consequently $p_A(b), b \in Y$ as argued above. Because $\nu(z) \in Y$ and $e(p_A(b),\nu(f(c)))+p_A(b)^{-1}\in Y$, we get that $\nu(f(c))\in Y$.\newline

\noindent Now suppose that $z\in \big[ e(p_A(b),\nu(f(c))),\nu(f(c))\big)$ and $a_1=p_A(b)$. It follows immediately that $a=p_A(b)$ and $\nu(f(c))-z\in \mathfrak{m}_{p_A(b)}$. Since $\lambda(f(c) - \nu(f(c)))^{-1} = p_A(b)$, $\lambda((f(c)-z)^{-1})$ is either $p_A(b)$ or $p_A(p_A(b))$. Thus $b$ and $p_A(b)$ are in $Y$. Because $e(p_A(b),\nu(f(c)))\leq z<e(p_A(b),\nu(f(c)))+p_A(b)^{-1}$, we can conclude that $e(p_A(b)),\nu(f(c)) = e(p_A(b)),\nu(z))\in Y$. Hence $\nu(f(c))\in Y$.\newline

\noindent Suppose that $z\in \big[ e(p_A(b),\nu(f(c))),\nu(f(c))\big)$ and $a_1>p_A(b)$. Again it follows easily that $a=a_1$ and  $\nu(f(c))-z\in \mathfrak{m}_{a}$. Since $\lambda((f(c) - \nu(f(c)))^{-1}) = a$, we get that $s_A(a)^{-1} \leq f(c)-z < p_A(a)^{-1}.$ Hence $\lambda((f(c)-z)^{-1})$ is either $a$ or $p_A(a)$. Thus $a \in Y$. Because $\nu(f(c))-z\in \mathfrak{m}_{a}$, $\nu(f(c)) - p_A(a)^{-1} < z < \nu(f(c))$. Since $\nu(f(c)) = e(p_A(b),\nu(f(c))) + b^{-1}$, we have $\nu(f(c)) - p_A(a)^{-1} = e(p_A(b),\nu(f(c)) + b^{-1}- p_A(a)^{-1}$. Because of $a \geq p_A(b)$ we can deduce from Axiom T\ref{axiom:succc} and Lemma \ref{lem:allones} that
\[
e(p_A(a),\nu(f(c))) = e(p_A(b),\nu(f(c)) + b^{-1}- p_A(a)^{-1}=\nu(f(c)) - p_A(a)^{-1}.
\]
Since $\nu(f(c)) - p_A(a)^{-1} < z < \nu(f(c))$, we have that $e(p_A(a),\nu(z)) = \nu(f(c)) - p_A(a)^{-1}$ by Axiom T\ref{axiom:ca2}. From $p_A(a) \in Y$ we conclude that $e(p_A(a),\nu(z)) \in Y$. Hence $\nu(f(c))\in Y$.\newline

\noindent The case that $z\in [ e(p_A(b),\nu(f(c)))+p_A(b)^{-1}-b^{-1},\nu(f(c))+p_A(b)^{-1}\big)$ can be handled similarly to the last two cases. We leave the details to the reader.
\end{proof}

\subsection*{Interaction between elements of $A$ and substructures} We finish this section with two Lemmas on the interplay of elements of $A$ and $\Cal L_C^*$-substructures. These results will be used in the next section.

\begin{lem}\label{lem:aspecial} Let $X \unlhd \Cal M$ and $a \in A$. If $[a]=[x]$ for some $x\in X$, then $a \in A(X)$.
\end{lem}
\begin{proof} By Lemma \ref{lem:llq} we can assume that $0\ll a$. Let $x\in X$ such that $[a_i]=[x]$. Since $X\unlhd \Cal M$, we get $a_i \in A(X)$ by Corollary \ref{cor:lambda}.
\end{proof}

\begin{lem}\label{lem:inx} Let $X\unlhd\Cal M$. Let $c=(c_1,\dots, c_n) \in C(X)^n$, $(a_1,\dots,a_n) \in A$ and $p,q \in \Q^n$ such that
\[
\mu_{(p,q)}\big(a_1^{-1},\dots, a_n^{-1},c_1-e(a_1,c_1),\dots,c_n-e(a_n,c_n)\big) \in A(X).
\]
Then $\big(a_1^{-1},\dots, a_n^{-1},c_1-e(a_1,c_1),\dots,c_n-e(a_n,c_n)\big)$ is $\Q$-linearly dependent over $C(X)$.
\end{lem}
\begin{proof} Set $b=\mu_{(p,q)}\big(a_1^{-1},\dots, a_n^{-1},c_1-e(a_1,c_1),\dots,c_n-e(a_n,c_n)\big)$.
Because $b\in A(X)$ and $X\unlhd \Cal M$, $p_A(b) \in A(X)$ and $c_i-e(p_{A}(b),c_i) \in X$. By Corollary \ref{cor:ainv} we have that for all $a\in A$
\[
a^{-1} - e(p_{A}(b),a^{-1}) = \left\{
                                     \begin{array}{ll}
                                       a^{-1}, & \hbox{if $p_A(b) \leq a$;} \\
                                       p_A(b)^{-1}, & \hbox{otherwise.}
                                     \end{array}
                                   \right.
\]
By Lemma \ref{lem:axiommu} and Lemma \ref{lem:eab}
\begin{align*}
&\sum_{i=1}^{n} p_i a_i^{-1} +\sum_{i=1}^{n}  q_i \big(c_i -e(a_i,c_i)\big)\\
 &= \sum_{i=1}^n p_i \big(a_i^{-1} - e(p_{A}(b),a_i^{-1})\big) + \sum_{i=1}^{n} q_i \big(c_i - e(a_i,c_i) - e(p_{A}(b),c_i - e(a_i,c_i))\big)\\
&=\sum_{p_A(b)\leq a_i} p_i a_i^{-1} + \sum_{p_A(b) > a_i} p_i p_{A}(b)^{-1} + \sum_{i=1}^{n} q_i (c_i - e(p_{A}(b),c_i)).
\end{align*}
The statement of the Lemma follows.
\end{proof}

\section{Quantifier elimination}

In this section we prove Theorem \ref{thm:complete} and Theorem \ref{thm:qe}. The actual proof will use several embedding lemmas that we will establish first. Let $\kappa = |\Cal L_C^+|$ and let $\Cal M,\Cal N\models \TT^+$ such that $|\Cal M|\leq \kappa$ and $\Cal N$ is $\kappa^+$-saturated. Let $X\unlhd \Cal M$ and suppose that $\beta: X \to \Cal N$ is an $\Cal L_C^*$-embedding.


\subsection*{Types of elements of $A$} We first consider types of elements of $A(M)$ over $X$.

\begin{lem}\label{lem:aspecial2} Let $a=(a_1,\dots,a_n) \in A(M)^n$ and $b=(b_1,\dots,b_n) \in A(N)^n$.
Then
\[
\beta \tp_{\Cal L_B}(a|X)=\tp_{\Cal L_B}(b|\beta(X)) \Rightarrow \beta \tp_{\Cal L}(a|X)=\tp_{\Cal L}(b|\beta(X)).
\]
\end{lem}
\begin{proof} By Lemma \ref{lem:llq} we can assume that $0\ll a_i$ and $0\ll b_i$ for $i=1,\dots, n$. By Corollary \ref{cor:oneaperlevel} and Lemma \ref{lem:leveltype} it is enough show that the statement of the Lemma holds for $n=1$. Let $a \in A(M)$ and $b\in A(N)$ be such that $\beta \tp_{\Cal L_B}(a|X)=\tp_{\Cal L_B}(b|\beta(X))$. We can immediately reduce to the case that $a\notin X$ and $b\notin \beta(X)$.
It is left to show that $b$ lies in the image of the cut of $a$ over $X$ under $\beta$. Suppose there are $x,y \in X$ such that $x<a<y$. Because $X$ is closed under $\lambda,p_A$ and $s_A$, $s_A(\lambda(x))\in X$ and $\lambda(y)\in X$. Since $a\notin X$, we get that $s_A(\lambda(x)) < a < \lambda(y)$. Because $b$ satisfies $\beta \tp_{\Cal L_B}(a/X)$ and $\beta$ is a $\Cal L_C^*$-embedding, we conclude that $s_A(\lambda(\beta(x))) < b < \lambda(\beta(y))$. Hence $\beta(x)<b<\beta(y)$.
\end{proof}

\begin{cor}\label{cor:aspecial} Let $c \in M$, $d\in N$, $a \in A(M)\setminus A(X)$ and $b\in A(N)\setminus A(\beta(X))$ such that $[c]=[a]$ and $[d]=[b]$. Then
 \[
\beta \tp_{\Cal L_B}(a|X) = \tp_{\Cal L_B}(b|\beta(X)) \Rightarrow \beta \tp_{\Cal L}(c|X) = \tp_{\Cal L}(d|\beta(X)).
\]
\end{cor}
\begin{proof} Since $a \notin A(X)$ and $[a]=[c]$, there is no $x \in X$ such that $[x]=[c]$ by Lemma \ref{lem:aspecial}. Therefore $x<a<y$ iff $x<c< y$ for all $x,y \in X$. Similarly one shows that $\beta(x)<b<\beta(y)$ iff $\beta(x)<d< \beta(y)$ for all $x,y \in X$. By Lemma \ref{lem:aspecial2} $\beta \tp_{\Cal L}(a|X)=\tp_{\Cal L}(b|\beta(X))$. Thus $\beta \tp_{\Cal L}(c|X) = \tp_{\Cal L}(d|\beta(X))$.
\end{proof}

\noindent Later we will have to extend $X$ not only by elements of $A(M)$, but also by their images under $e(-,z)$ for certain $z \in C(M)$. The next Lemma shows that for $a\in A^n$ the $\Cal L_B$-type of $a$ over $X$ does not only determine the $\Cal L$-type of $a$, but also the $\Cal L$-type of $e(a,c)$ and $a$ over $X$ for every $c\in C(X)$.

\begin{prop}\label{prop:abimpl} Let $c=(c_1,\dots, c_n) \in C(X)^n$ and $a=(a_1,\dots,a_n)\in A(M)^n,b=(b_1,\dots,b_n)\in A(N)^n$.
 If $\beta \tp_{\Cal L_B}(a|X)=\tp_{\Cal L_B}(b|\beta(X))$, then
\[
\beta \tp_{\Cal L}(a,e(a_1,c_1),\dots,e(a_n,c_n)|X)=\tp_{\Cal L}(b,e(b_1,\beta(c_1)),\dots,e(b_n,\beta(c_n))|\beta(X)).
\]
\end{prop}
\begin{proof}   By Lemma \ref{lem:llq} we can assume that $0\ll a_i$ and $0\ll b_i$ for $i=1,\dots, n$. We can easily reduce to the case that $a_i \neq a_j$ and $b_i \neq b_j$ for $i,j \in \{1,\dots,n\}$ and $i\neq j$.
By Corollary \ref{cor:oneaperlevel} $[a_i] \neq [a_j]$ and $[b_i] \neq [b_j]$ for $i\neq j$. Suppose there is $x\in X$ such that $[a_i]=[x]$. Since $X \unlhd \Cal M$, we get $a_i \in A(X)$ by Corollary \ref{cor:lambda}. Thus $b_i \in A(\beta(X))$, because $b_i$ satisfies $\beta \tp_{\Cal L_B}(a_i|X)$. Therefore we can assume that there is no $x \in X$ with $[a_i]=[x]$. Similarly we can also assume that there is no $x\in X$ with $[b_i]=[\beta(x)]$.\\
After reordering $c_1,\dots,c_n$, we can assume that there is $l\in \N$ such that $a_1^{-1}$,$\dots$, $a_n^{-1},c_1-e(a_1,c_1),\dots,c_l - e(a_l,c_l)$ are $\Q$-linearly independent over $C(X)$, and for all $k>l$, $c_k-e(a_k,c_k)$ is $\Q$-linearly dependent over $C(X)$ and $a_1^{-1},\dots a_n^{-1},c_1-e(a_1,c_1),\dots,c_l - e(a_l,c_l)$.
Because  $a_i^{-1}$ and $c_i-e(a_i,c_i)$ are $\Cal L_B$-definable over $X\cup \{a_i\}$ for each $i$ and $\beta \tp_{\Cal L_B}(a|X)=\tp_{\Cal L_B}(b|\beta(X))$, we deduce from Proposition \ref{prop:buechiexp}(v) that
$c_k-e(b_k,c_k)$ satisfies the same $\Q$-linear dependency over $C(\beta(X))$ and $b_1^{-1},\dots b_n^{-1},\beta(c_1)-e(b_1,\beta(c_1)),\dots,\beta(c_l) - e(b_l,\beta(c_l))$.
%
It is only  left to show that
\[
\beta \tp_{\Cal L}(a,e(a_1,c_1),\dots,e(a_l,c_l)|X)=\tp_{\Cal L}(b,e(b_1,\beta(c_1)),\dots,e(b_l,\beta(c_l))|\beta(X)).
\]
We will use the following abbreviations
\begin{align*}
u &:= \big(a_1^{-1},\dots,a_n^{-1},c_1-e(a_1,c_1),\dots,c_l-e(a_{l},c_l)\big),\\
v &:= \big(b_1^{-1},\dots,b_n^{-1},\beta(c_1)-e(b_1,\beta(c_1)),\dots,\beta(c_l)-e(b_{l},\beta(c_l))\big).
\end{align*}
Because $c_i \in X$ for each $i$, it is enough to show that $\beta \tp_{\Cal L}(u|X)=\tp_{\Cal L}(v|\beta(X))$. Since $\beta \tp_{\Cal L_B}(a|X)=\tp_{\Cal L_B}(b|\beta(X))$, we have $\beta \tp_{\Cal L_B}(u|X)=\tp_{\Cal L_B}(v|\beta(X))$ by Lemma \ref{lem:definv}. By $\Q$-linear independence of $u$ and Lemma \ref{lem:distinctskies2} there are tuples of rational numbers $r_1,\dots,r_{n+l}\in\Q^{n+l}$ with $r_{i,i}\neq 0$ and $r_{i,j}=0$ for $j>i$ such that $0\neq \mu_{r_i}(u) \neq \mu_{r_j}(u)$ for $i\neq j$. Because $\mu_r$ is a $\Cal L_B$-definable function for every $r\in \Q^m$ and $\beta \tp_{\Cal L_B}(u|X)=\tp_{\Cal L_B}(v|\beta(X))$, we have that $0\neq\mu_{r_i}(v) \neq \mu_{r_j}(v)$ for $i\neq j$. Moreover, from Lemma \ref{lem:inx} and $u$ being $\Q$-linearly independent over $C(X)$ we conclude that $\mu_{r_i}(u)\notin X$ for each $i$. Since $\beta \tp_{\Cal L_B}(u|X)=\tp_{\Cal L_B}(v|\beta(X))$, we get that $\mu_{r_i}(v)\notin \beta(X)$ for each $i$. This implies that $0\ll \mu_{r_i}(u)$ and $0\ll \mu_{r_i}(v)$. By Lemma \ref{lem:valsum} $[|r_i\cdot u|^{-1}]=[p_A(\mu_{r_i}(u))]$ and $[|r_i \cdot v|^{-1}]=[p_A(\mu_{r_i}(v))]$. Since $\beta \tp_{\Cal L_B}(u|X)=\tp_{\Cal L_B}(v|\beta(X))$, $\beta \tp_{\Cal L_B}(p_A(\mu_{r_i}(u))|X)=\tp_{\Cal L_B}(p_A(\mu_{r_i}(v))|\beta(X))$. By Corollary \ref{cor:aspecial} $\beta \tp_{\Cal L}(r_i\cdot u|X) = \tp_{\Cal L}(r_i \cdot v|\beta(X))$. Hence by Lemma \ref{lem:leveltype}
\[
\beta\tp_{\Cal L}(r_1\cdot u,\dots, r_n \cdot u|X) = \tp_{\Cal L}(r_1\cdot v,\dots, r_n \cdot v|\beta(X)).
\]
Because $r_{i,i}\neq 0$ for each $i$ and $r_{i,j}=0$ for $j>i$, $\beta \tp_{\Cal L}(u|X)=\tp_{\Cal L}(v|\beta(X))$.
\end{proof}

\subsection*{Types of elements of $C$} Now consider types of elements of $C(M)$ over $X$. In contrast to the results about types of tuples of elements of $A$, we will only consider types of a single element of $C(M)$.

\begin{lem}\label{lem:ctype1} Let $c \in C(M)$ and $d\in C(N)$. Then
\[
\beta \tp_{\Cal L_B}(c|X) = \tp_{\Cal L_B}(d|\beta(X)) \Rightarrow \beta \tp_{\Cal L}(c|X)=\tp_{\Cal L}(d|\beta(X)).
\]
\end{lem}
\begin{proof} Suppose $\beta \tp_{\Cal L_B}(c|X) = \tp_{\Cal L_B}(d|\beta(X))$. We can easily reduce to the case that $c\notin X$ and $d \notin \beta(X)$. Since $X\unlhd \Cal M$ and $\beta$ is a $\Cal L_C^*$-embedding, both $X$ and $\beta(X)$ are closed under $\nu$. Consequently, for all $x \in X$ with $\nu(x)<c$, we have $x < c$, and for all $y \in X$ with $c<y$, we have $c < \nu(y)$. Similarly for all $x \in X$ with $\nu(\beta(x))<d$, we have $\beta(x) < d$, and for all $y \in X$ with $d<\beta(y)$, we have $d < \nu(\beta(y))$. Let $x,y \in X$. Since $\beta \tp_{\Cal L_B}(c|X) = \tp_{\Cal L_B}(d|\beta(X))$, $x<c<y$ iff $\nu(x)<c<\nu(y)$ iff $\nu(\beta(x)) < d<\nu(\beta(y))$ iff $\beta(x)<d<\beta(y)$. Thus $d$ lies in the image of the cut of $c$ over $X$ under $\beta$. Therefore $\beta \tp_{\Cal L}(c|X)=\tp_{\Cal L}(d|\beta(X))$.
%
\end{proof}

\begin{cor}\label{cor:ctype1} Let $\varphi(x)$ be a $\Cal L$-$X$-formula and let $p(x)$ be a complete $\Cal L_B$-type over $X$. Then there is a $\Cal L_B$-$X$-formula $\psi(x)\in p(x)$ such that either
$\Cal M \models \forall x \in C \ \psi(x) \rightarrow \varphi(x)$  or
$\Cal M \models \forall x \in C \ \psi(x) \rightarrow \neg \varphi(x)$ .
\end{cor}
\begin{proof} By replacing $\Cal M$ by a $\kappa^+$-saturated elementary extension, we can assume that $\Cal M=\Cal N$. Suppose the conclusion of the corollary fails. By saturation of $\Cal M$ there are $c,d \in C(M)$ such that $p=\tp_{\Cal L_B}(c|X)=\tp_{\Cal L_B}(d|X)$ and $\varphi(c) \wedge \neg \varphi(d)$. By Lemma \ref{lem:ctype1} $\tp_{\Cal L}(c|X) = \tp_{\Cal L}(d|X)$. This contradicts $\varphi(c) \wedge \neg \varphi(d)$. 
\end{proof}

\begin{prop}\label{prop:cbimpl} Let $c \in C(M), d\in C(N)$ and $a_1,\dots,a_n \in A(X)$. If $\beta \tp_{\Cal L_B}(c|X) = \tp_{\Cal L_B}(d|\beta(X))$,
then
\begin{equation}\label{eq:lbrl}
\beta \tp_{\Cal L}(c,e(a_1,c),\dots,e(a_n,c)|X)=\tp_{\Cal L}(d,e(\beta(a_1),d),\dots,e(\beta(a_n),d)|\beta(X)).
\end{equation}
\end{prop}
\begin{proof} Let $c \in C(M)$, $d\in C(N)$ such that $\beta \tp_{\Cal L_B}(c|X) = \tp_{\Cal L_B}(d|\beta(X))$. Let $a_1,\dots,a_n\in A(X)$ with $a_1<a_2<\dots < a_n$.
Because $e(x,y)$ is $\Cal L_B$-definable from $x,y$ and $\beta \tp_{\Cal L_B}(c|X) = \tp_{\Cal L_B}(d|\beta(X))$,
\[
\beta \tp_{\Cal L_B}(c,e(a_1,c),\dots,e(a_n,c)|X)=\tp_{\Cal L_B}(d,e(\beta(a_1),d),\dots,e(\beta(a_n),d)|\beta(X)).
\]
We will now show that the statement by induction on $n$. By Proposition \ref{prop:buechiexp}(v) we can directly reduce to the case that $e(a_1,c),\dots,$ $e(a_n,c)$ are $\Q$-linearly independent over $C(X)$ and $e(\beta(a_1),d),\dots,e(\beta(a_n),d)$ are $\Q$-linearly independent over $C(\beta(X))$. For $n=0$, note that  $\tp_{\Cal L}(c|X)=\tp_{\Cal L}(d|\beta(X))$ by Lemma \ref{lem:ctype1}. For the induction step, suppose that
\begin{equation}\label{eq:lbrl3}
\beta \tp_{\Cal L}(c,e(a_1,c),\dots,e(a_{n-1},c)|X)=\tp_{\Cal L}(d,e(\beta(a_1),d),\dots,e(\beta(a_{n-1}),d)|\beta(X)).
\end{equation}
We will use the following abbreviations:
\[
u := \big(e(a_1,c),\dots,e(a_{n-1},c)\big), v:= \big(e(\beta(a_1),d),\dots,e(\beta(a_{n-1}),d)\big).
\]
\noindent Suppose that $\overline{u}^{a_n},\overline{c}^{a_n}$ are $\dcl$-dependent over $\overline{X}^{a_n}$. By Lemma \ref{lem:eavf12} and since $e(a_n,e(a_j,c))=e(a_j,c)$ for $j\leq n$, there are $x \in C(X)^m$, $p \in \Q^m$ and $q\in \Q^n$ such that $e(a_n,c) = (p,q) (x,u)$. This contradicts our assumption that $e(a_1,c),\dots,\linebreak e(a_n,c)$ are $\Q$-linearly independent over $C(X)$. Similarly we can rule out that
$\overline{v},\overline{d}^{a_n}$ are $\dcl$-dependent over $\overline{\beta(X)}^{a_n}$.

\noindent Now suppose that $\overline{u}^{a_n},\overline{c}^{a_n}$ are $\dcl$-independent over $\overline{X}^{a_n}$. From Lemma \ref{lem:definv} and the fact that $\beta \tp_{\Cal L_B}(c|X) = \tp_{\Cal L_B}(d|\beta(X))$, we deduce $\beta \tp_{\Cal L_B}(c- e(a_n,c)|X)=\tp_{\Cal L_B}(c-e(a_n,d)|\beta(X))$. Thus by Lemma \ref{lem:ctype1}
\begin{equation}\label{eq:lbrl4}
\beta \tp_{\Cal L}(c- e(a_n,c)|X)=\tp_{\Cal L}(d-e(a_n,d)|\beta(X)) .
\end{equation}
Let $Z \subseteq M^{m+n}$ open and $f : Z \to M$ and $g: Z \to M$ be $\Cal L$-$\emptyset$-definable continuous functions. Let $x=(x_1,\dots,x_m) \in X^m$. By o-minimality of $T$ it is enough to show that
\begin{equation}\label{eq:lbrl2}
f(x,c,u) < c- e(a_n,c) < g(x,c,u) \hbox{ iff } f(\beta(x),d,v) < d-e(\beta(a_n),d) < g(\beta(x),d,v).
\end{equation}
By regular cell-decomposition in o-minimal structures and \eqref{eq:lbrl4}, we can reduce to the case that $f,g$ are not constant in all the last $n$ coordinates. By our assumptions on $c$ and $d$ we have that $f(x,c,u)$, $g(x,c,u) \notin \mathfrak{m}_{a_n}$. However, by Axiom T\ref{axiom:ca} $c-e(a_n,c)\in \mathfrak{m}_{a_n}$. Therefore
\[
f(x,c,u) < c- e(a_n,c)< g(x,c,u) \hbox{ iff } f(x,c,u) < 0 < g(x,c,u).
\]
A similar argument shows that
\[
f(\beta(x),d,v) < d- e(\beta(a_n),d)< g(\beta(x),d,v) \hbox{ iff } f(\beta(x),d,v) < 0 < g(\beta(x),d,v).
\]
Hence \eqref{eq:lbrl2} follows from \eqref{eq:lbrl3}.
\end{proof}

\begin{lem}\label{lem:yspecial} Let $U \subseteq C(M), V \subseteq C(N)$ and $\gamma : \dcl(X\cup U) \to \dcl(\beta(X)\cup V)$ such that
\begin{itemize}
\item[(i)] $\dcl(X \cup U)\unlhd \Cal M$,
\item[(ii)] $\gamma$ is an $\Cal L$-isomorphism extending $\beta$ with $\gamma(U)=V$,
\item[(iii)] $\beta \tp_{\Cal L_B}(U|X) = \tp_{\Cal L_B}(V|\beta(X))$.
\end{itemize}
Then $\dcl(\beta(X)\cup V) \unlhd \Cal N$ and $\gamma$ is $\Cal L_C^{*}$-isomorphism.
\end{lem}
\begin{proof} For ease of notation, set $X':=\dcl(X\cup U)$ and $Y':=\dcl(\beta(X) \cup V)$. We will establish the conclusion of the Lemma by proving a sequence of claims.

\begin{claim}\label{claim:one} Let $c \in C(X')$.  Then $c$ is $\Cal L_B$-definable over $X \cup U$, $\gamma(c) \in C(Y')$ and $\gamma \tp_{\Cal L_B}(c|X\cup U) = \tp_{\Cal L_B}(\gamma(c)|\beta(X)\cup V)$.
\end{claim}
\begin{proof}[Proof of Claim \ref{claim:one}] Since $X\unlhd \Cal M$, there are $p\in \Q^m$, $q \in \Q^n$ $x \in X^m$ and $u \in U^n$ such that $c = p\cdot x + q \cdot u$ by Lemma \ref{lem:eavf02}. By (ii) $\gamma(c) = p \cdot \beta(x) + q \cdot \gamma(u)$ and $\gamma(u) \in V^n$. By (iii) and Proposition \ref{prop:buechiexp}(vi) $p\cdot x + q \cdot u\in C(M)$ iff $p \cdot \gamma(x) + q \cdot \beta(u) \in C(N)$. Hence $\gamma(c) = p \cdot \gamma(z) + q \cdot \gamma(u) \in C(Y')$.
\end{proof}

\noindent Because every element in $C(X')$ is $\Cal L_B$-definable over $X\cup U$ by Claim \ref{claim:one}, we can conclude that $\gamma \tp_{\Cal L_B}(C(X')|X\cup U) = \tp_{\Cal L_B}(\gamma(C(X'))|\beta(X)\cup V)$.

\begin{claim}\label{claim:two} Let $a \in A(X')$. Then $a$ is $\Cal L_B$-definable over $X \cup U$, $\gamma(a) \in A(Y')$ and $\gamma \tp_{\Cal L_B}(a|X\cup U) = \gamma \tp_{\Cal L_B}(\gamma(a)|\beta(X)\cup V)$.
\end{claim}
\begin{proof} Since $a \in A(X')$, $a^{-1} \in C(X')$. Thus $\gamma(a^{-1})=\gamma(a)^{-1}$ by (ii) and $\gamma(a)^{-1} \in C(Y')$ by Claim \ref{claim:one}. Since $a \in A(X')$, there is a unique $b \in A(M)$ such that $E(b,a^{-1})$. By  Claim \ref{claim:one}, $\gamma \tp_{\Cal L_B}(a^{-1}|X \cup U)= \tp_{\Cal L_B}(\gamma(a)^{-1}|\beta(X) \cup V)$. Therefore there is also a unique $b \in A(N)$ such that $E(b,\gamma(a)^{-1})$. The existence of such a $b$ implies that $\gamma(a) \in A(Y')$. By Lemma \ref{lem:definv}(i)  $\gamma \tp_{\Cal L_B}(a|X \cup U)= \tp_{\Cal L_B}(\gamma(a)|\beta(X) \cup V)$.
\end{proof}

\begin{claim}\label{claim:three} $A(Y') = \gamma(A(X'))$ and $\lambda(\gamma(x))=\gamma(\lambda(x))$ for all $x\in X'$.
\end{claim}
\begin{proof} We first prove the second statement. Let $y \in Y'$ and $x \in X'$ such that $\gamma(x)=y$. By Claim \ref{claim:two} we have that $\gamma(\lambda(x))\in A(Y')$, $\gamma(s_A(\lambda(x))\in A(Y')$ and $s_A(\gamma(\lambda(x))= \gamma(s_A(\lambda(x)))$. Hence $\gamma(\lambda(x)) \leq \gamma(x)=y < s_A(\gamma(\lambda(x)))$. Consequently $\lambda(y)= \gamma(\lambda(x)) \in A(Y')$. For the proof of the first statement let $a \in A(Y')$ and $b \in X'$ such that $\gamma(b)=a$. Then $a=\lambda(a)=\lambda(\gamma(b))=\gamma(\lambda(b))$. Thus $\lambda(b)=b$ and $b \in A(X')$.
\end{proof}

\noindent It follows immediately from Claim \ref{claim:three} that $Y'$ is closed under $\lambda$.

\begin{claim}\label{claim:four} $C(Y') = \gamma(C(X'))$ and $\nu(\gamma(x))=\gamma(\nu(x))$ for all $x \in X'$.
\end{claim}
\begin{proof} Again we prove the second statement first. Let $y \in Y'$ and $x \in X'$ such that $\gamma(x)=y$. We can reduce to the case that $y \in (0,1)$, $\nu(y)\neq 0$ and $\nu(y)\neq y$. Let $d \in A(M)$ be maximal such that $\neg E(d,\nu(x))$. By Corollary \ref{cor:nuexplicit}
\[
x \in \Big[ e(p_A(d),\nu(x))+d^{-1}, e(p_A(d),\nu(x))+p_A(d)^{-1}-d^{-1}\Big).
\]
Because $d$ is $\Cal L_B$-definable from $\nu(x)$ and $X'\unlhd \Cal M$, $d \in A(X')$ and so is $p_A(d)$. Since $X'$ is closed under $e$, $e(p_A(d),\nu(x))\in X'$.
By Claim \ref{claim:two}, $\gamma(d) \in A(Y')$. By Claim \ref{claim:one} $\gamma(\nu(x)),\gamma(e(p_A(d),\nu(x)) \in C(Y')$ and $\gamma(e(p_A(d),\nu(x)))=e(p_A(\gamma(d)),\gamma(\nu(x)))$.
Since $\gamma$ is an $\Cal L$-isomorphism, we have
\[
y \in \big [(e(p_A(\gamma(d)),\gamma(\nu(x))))+\gamma(d)^{-1}, e(p_A(\gamma(d)),\gamma(\nu(x)))+p_A(\gamma(d))^{-1}-\gamma(d)^{-1}\big ).
\]
By Lemma \ref{lem:complintervals} $\nu(y) = e(p_A(\gamma(d)),\gamma(\nu(x)))+\gamma(d)^{-1} \in C(Y')$. Hence  $\nu(y)=\gamma(\nu(x))$. For the proof of the first statement let $y \in C(Y')$ and $x \in X'$ such that $\gamma(x)=y$. Then $\gamma(\nu(x)) = \nu(y) = y$. Since $\gamma$ is bijective, $\nu(x)=x$. Therefore $x\in C(X')$.
\end{proof}

\noindent We directly get from Claim \ref{claim:four} that $Y'$ is closed under $\nu$. Combining Claim \ref{claim:four} with the statement after Claim \ref{claim:one} we get that $\beta \tp_{\Cal L_B}(C(X') | X) = \tp_{\Cal L_B}(C(Y') | \beta(X))$. Hence $\gamma$ is a $\Cal L_C^B$-isomorphism. Since $X' \unlhd \Cal M$, it follows easily that $Y'$ is closed under $e$ and under all $\Cal L_C$-definable functions into $A$, and that these functions commute with $\gamma$. Since we already know that $Y'$ is closed under $\lambda$ and $\nu$, and that $\gamma$ commutes with $\nu$ and $\lambda$, we have that $Y'\unlhd \Cal N$ and $\gamma$ is $\Cal L_C^*$-isomorphism.
\end{proof}

\subsection*{Embedding Lemmas} We will now prove the necessary embedding lemmas for our quantifier elimination result. We still assume that $\kappa = |\Cal L_C^+|$ and $\Cal M,\Cal N\models \TT^+$ are such that $|\Cal M|\leq \kappa$ and $\Cal N$ is $\kappa^+$-saturated.

\begin{defn} Let $Z\subseteq M$. We define $\Cal D(Z)$ as the union of $\lambda(Z)$ and the set of all elements in $A(M)$ that are $\Cal L_B$-definable from $Z$.
\end{defn}

\begin{prop}\label{prop:extendbyc} Let $X\unlhd \Cal M$, $\beta: X \to \Cal N$ be a $\Cal L_C^*$-embedding and $c \in C(M)$. Then there is a $\Cal L_C^*$-embedding $\gamma$ into $\Cal N$ extending $\beta$ such that $c \in \dom(\gamma)$ and $\dom(\gamma)\unlhd \Cal M$.
\end{prop}
\begin{proof}
Set $U_0 := A(X)$ and  $V_0 := \emptyset$, and recursively define
\begin{align*}
U_{i+1} &:= \Cal D(X \cup \{c\} \cup U_i\cup V_i), \ V_{i+1} := e(U_i,C(X)\cup \{c\}).
\end{align*}
Set $U := \bigcup_{i\in \N} U_i, V := \bigcup_{i \in \N} V_i$. Since $\beta$ is a $\Cal L_C^B$-isomorphism and $\Cal N$ is saturated, there is $W \subseteq A(N)$ such that
\begin{equation}\label{eq:eeaeq1}
\beta \tp_{\Cal L_B}(U | X) = \tp_{\Cal L_B}(W | \beta(X)).
\end{equation}
By Proposition \ref{prop:abimpl} we have
\begin{equation*}
\beta \tp_{\Cal L}(U \cup e(U,C(X)) | X) = \tp_{\Cal L}(W \cup e(W,C(\beta(X)) | \beta(X)).
\end{equation*}
Therefore $\beta$ extends to a $\Cal L$-isomorphism $\beta'$ between $X':=\dcl(X \cup U \cup e(U,C(X))$ and $Y':=\dcl(Y \cup W \cup e(W,C(\beta(X))))$ such that $\beta'(U)=W$ and $e(u,c)=e(\beta'(u),\beta(c))$ for all $u \in U$ and $c\in C(X)$.  By \eqref{eq:eeaeq1} $\beta \tp_{\Cal L_B}(U \cup e(U,C(X)) | X) = \tp_{\Cal L_B}(W \cup e(W,C(\beta(X)) | \beta(X))$.
By Lemma \ref{lem:yspecial} it is only left to show  that $X'\unlhd\Cal M$. We will prove that $X'$ satisfies the assumptions of Proposition \ref{cor:aclosure}. We first establish that $X'$ is $A$-closed. Note that by construction of $U$ and $V$, for every $x\in X'$ there is $i\in \N$ such that $x \in \dcl(X\cup U_i\cup V_i)$. Hence $\lambda(y) \in U_{i+1}$. Thus $A(X')=U$. If $a \in A(M)$ is $\Cal L_B$-definable from $X'$, there is $i\in \N$ such that $a$ is $\Cal L_B$-definable from $X \cup U_i \cup V_i$. Hence $a\in U_{i+1}$ and $X'$ is $A$-closed. Since $U \cup e(U,C(X))$ is closed under $e(a,-)$ for each $a \in U$ and $A(X')=U$, we get $X'\unlhd \Cal M$ by Proposition \ref{cor:aclosure}. Therefore $\beta'$ is a $\Cal L_C^*$-embedding.\\

\noindent We now extend $\beta'$ to a $\Cal L_C^*$-embedding $\gamma$ whose domain contains $c$. Because $\beta'$ is a $\Cal L_C^B$-embedding and $\Cal N$ is saturated, there is $d \in C(N)$ such that
\begin{equation}\label{eq:eeaeq2}
\beta' \tp_{\Cal L_B}(c | X') = \tp_{\Cal L_B}(d | Y').
\end{equation}
By Proposition \ref{prop:cbimpl} $\beta' \tp_{\Cal L}(c \cup e(A(X'),c) | X') = \tp_{\Cal L}(d \cup e(A(Y'),d)| Y')$. Consequently $\beta'$ extends to an $\Cal L$-isomorphism $\gamma$ between $X'':=\dcl(X' \cup \{c\} \cup e(A(X'),c))$ and $Y'':=\dcl(Y' \cup \{d\} \cup e(A(Y'),d))$ mapping $c$ to $d$ and $e(a,c)$ to $e(\beta'(a),d)$ for every $a\in A(X')$. By \eqref{eq:eeaeq2}  $\beta' \tp_{\Cal L_B} (c \cup e(A(X'),c)|X') = \tp_{\Cal L_B}(d \cup e(A(Y',d)|Y')$. In order to show that $\gamma$ is
 a $\Cal L_C^*$-embedding, it is again only left to show that $X''\unlhd \Cal M$.  We will establish that the conditions of Proposition \ref{cor:aclosure} are satisfied. We first prove that $X''$ is $A$-closed. By construction of $U$ and $V$, for every $x\in X''$ there is $i\in \N$ such that $x \in \dcl(X \cup \{c\} \cup U_i \cup V_i)$.
Thus for $x \in X''$ we can find $i\in \N$ such that $x \in \dcl(X \cup \{c\} \cup U_i \cup V_i)$. Hence $\lambda(x) \in U_{i+1}$. In particular, $A(X'')=U$. If $a \in A(M)$ is $\Cal L_B$-definable from $X''$, there is $i\in \N$ such that $a$ is $\Cal L_B$-definable from $X  \cup \{c\}\cup U_i \cup V_i$. Hence $a\in U_{i+1}$ and $X''$ is $A$-closed. Since $A(X'')=U=A(X')$, we also have that $\{c\} \cup e(A(X'),c)$ is closed under $e(a,-)$ for  each $a \in A(X'')$. By Proposition \ref{cor:aclosure} $X''\unlhd \Cal M$. Thus $\gamma$ is a $\Cal L_C^*$-embedding.
\end{proof}

\begin{cor}\label{cor:extendbymanyc}  Let $X\unlhd \Cal M$, $\beta: X \to \Cal N$ be a $\Cal L_C^*$-embedding and $Z \subseteq C(M)$ be such that $|Z|\leq \kappa$. Then there is a $\Cal L_C^*$-embedding $\gamma$ extending $\beta$ such that $Z \subseteq \dom(\gamma)$ and $\dom(\gamma) \unlhd \Cal M$.
\end{cor}

\noindent So far we have only extended $\beta$ to another $\Cal L_C^*$-embedding. In order to extend $\beta$ to a $\Cal L_C^+$-embedding, we need to better understand the interaction of $\beta$ with $\nu_f$ and $\tau_f$.

\begin{lem}\label{lem:nusandtaus} Let $X\unlhd \Cal M$ and $\beta: X \to \Cal N$ be a $\Cal L_C^*$-embedding. Let $x\in X^l$, $y\in M$, $f: U \subseteq M^{l+n} \to [0,1]$ be $\Cal L$-$\emptyset$-definable and continuous. Let $\gamma$ be a $\Cal L$-embedding extending $\beta$ with $y\in \dom(\gamma)$. If $\nu_f(x,y),\tau_f(x,y) \in X$, then $\beta(\nu_f(x,y))=\nu_f(\beta(x),\gamma(y))$ and $\beta(\tau_f(x,y))=\tau_f(\beta(x),\gamma(y))$.
\end{lem}
\begin{proof} Let $g_1,\dots,g_k : M^{l+n} \to M$ be as in Axiom T\ref{axiom:tau} for $f$. Suppose that $\nu_f(x,y),\tau_f(x,y) \in X$. Let $i,j\leq k$ such that
\[
g_i(x,\nu_f(x,y)) = \sup_{d \in C(M)^n\cap U_x, f(x,d)\leq y} f(x,d)
\]
and
\[
g_j(x,\tau_f(x,y)) = \inf_{d \in C(M)^n\cap U_x, f(x,d)\geq y} f(x,d).
\]
Let $\varphi(z,x,u,v)$ be the $\Cal L$-formula stating that $(x,z) \in U$ and one of the following three statements holds:
\begin{itemize}
\item[(i)] $f(x,z)= g_i(x,u)$ and $z$ is lexicographically smaller than $u$,
\item[(ii)] $f(x,z)= g_j(x,v)$ and $z$ is lexicographically smaller than $v$,
\item[(iii)] $g_i(x,u) < f(x,z) < g_j(x,v)$.
\end{itemize}
By definition of $\nu_f$ and $\tau_f$ and our choice of $i,j$, there is no $c\in C(M)$ with $\varphi(c,x,\nu_f(x,y),\tau_f(x,y))$. We now show that
there is no $d\in C(N)$ such that $\varphi(d,\beta(x),\beta(\nu_f(x,y)),\beta(\nu_f(x,y))$. Suppose towards a contradiction that there is such an $d \in C(N)$.
Let $p(z)$ be the $\Cal L_B$-type $\beta^{-1}\tp_{\Cal L_B}(d | \beta(X))$. By Corollary \ref{cor:ctype1} there is a $\Cal L_B$-formula $\psi(z,x')\in  p(z)$, where $x' \in X^n$ such that
\begin{equation}\label{eq:lplusproof}
\Cal M \models \forall z\in C(M) \ \psi(z,x') \rightarrow \varphi(c,x,\nu_f(x,y),\tau_f(x,y)).
\end{equation}
Since $\Cal N \models \psi(d,\beta(x'))$, $\Cal N\models \exists z \in C(N) \ \psi(z,\beta(x'))$. Because $\beta$ is a partial $\Cal L_C^*$-isomorphism and $\psi$ is an $\Cal L_B$-formula, $\Cal M\models \exists z \in C(M) \ \psi(z,x')$. Hence there is $d' \in C(M)$ such that $\psi(d',x)$. By \eqref{eq:lplusproof} we get $\varphi(d',x,\nu_f(x,y),\tau_f(x,y))$, a contradiction. Since $\gamma$ is a $\Cal L$-embedding,
\[
g_i(\beta(x),\beta(\nu_f(x,y))\leq \gamma(y)\leq g_j(\beta(x),\beta(\tau_f(x,y)).
\]
Because there is no $d\in C(N)$ with $\varphi(d,\beta(x),\beta(\nu_f(x,y)),\beta(\nu_f(x,y))$, we can conclude that $\beta(\nu_f(x,y))=\nu_f(\beta(x),\gamma(y))$ and $\beta(\tau_f(x,y))=\tau_f(\beta(x),\gamma(y))$.
\end{proof}

\begin{prop}\label{prop:extendbycplus} Let $X\unlhd^+ \Cal M$, $\beta: X \to \Cal N$ be a $\Cal L_C^+$-embedding and $c \in C(M)$. Then there is a $\Cal L_C^+$-embedding $\gamma$ into $\Cal N$ extending $\beta$ such that $c \in \dom(\gamma)$.
\end{prop}
\begin{proof} By Proposition \ref{prop:extendbyc} we can extend $\beta$ to a $\Cal L_C^*$-embedding $\gamma : X' \to Y'$ such that $c \in X'$ and $X'\unlhd \Cal M$. For $Z\subseteq M$, we define $\Cal E(Z)$ to be the union of all $\nu_f(Z^l,Z) \cup \tau_f(Z^l,Z)$  for all continuous $\Cal L$-$\emptyset$-definable $f: U \subseteq M^{l+n} \to [0,1]$. Since $\kappa > |\Cal L|$, we have that $|\Cal E (Z)|\leq \kappa$ if $|Z|\leq \kappa$. Hence by Corollary \ref{cor:extendbymanyc} we can extend $\gamma$ to a $\Cal L_C^*$-embedding $\gamma_0 : X_0 \to Y_0$ such that $\Cal E(X') \subseteq X_0$ and $X_0 \unlhd \Cal M$. For $n\in \N$, define $\gamma_{n+1} : X_{n+1} \to Y_{n+1}$ to be a $\Cal L_C^*$-embedding extending $\gamma_n$ such that $\Cal E(X_n) \subseteq X_{n+1}$ and $X_{n+1} \unlhd \Cal M$. Let $\gamma_{\infty} = \bigcup_{n=0}^{\infty} \gamma _n : X_{\infty} \to Y_{\infty}$. Because each $\gamma_n$ is a $\Cal L_C^*$-embedding, so is $\gamma_{\infty}$. Moreover, since $X_{n}\unlhd \Cal M$ for each $n$, $X_{\infty}\unlhd \Cal M$. It is easy to see that by the construction of $\gamma_{\infty}$, $X_{\infty}$ is closed under all $\nu_f$'s and $\tau_f$'s. Hence $X_{\infty} \unlhd^+ \Cal M$.  By Lemma \ref{lem:nusandtaus} $\gamma_{\infty}$ is a $\Cal L_C^+$-embedding.
\end{proof}

\noindent We are now ready to prove the two of the main results of the paper: quantifier elimination for $\TT^+$ and completeness of $\TT$.

\begin{proof}[Proof of Theorem \ref{thm:qe}] Let $\kappa = |\Cal L_C^+|$ and let $\Cal M,\Cal N\models \TT^+$ such that $|\Cal M|\leq \kappa$ and $\Cal N$ is $\kappa^+$-saturated. Let $X\unlhd^+ \Cal M$ and suppose that $\beta: X \to \Cal N$ is a $\Cal L_C^+$-embedding. It is enough to show that $\beta$ can be extended. By Proposition \ref{prop:extendbycplus} we can assume $C(M)\subseteq X$. Let $u \in \Cal M$. Without loss of generality we can assume that $u\in [0,1]$. We will extend $\beta$ to a $\Cal L_C^+$-embedding $\gamma$ such that $u \in \dom(\gamma)$. Let $v \in \Cal N$ such that $\tp_{\Cal L}(v|\beta(X))=\beta \tp_{\Cal L}(u|X)$. Because $\beta$ is a $\Cal L$-embedding, such a $v$ exists. Then $\beta$ extends to an $\Cal L$-embedding $\gamma$ between $\dcl(X\cup \{u\})$ and $\dcl(\beta(X) \cup \{v\})$ with $\gamma(u)=v$. Since $C(M)\subseteq X$, it is easy to check that $\dcl(X\cup \{u\})\unlhd^+ \Cal M$.
It is left to show that $\dcl(\beta(X) \cup \{v\})\unlhd^+ \Cal N$ and that $\gamma$ is a $\Cal L_C^+$-embedding. Because $\beta$ is also a $\Cal L_C^+$-embedding, it is enough to prove that $\dcl(\beta(X) \cup \{v\})\cap C(N) \subseteq \beta(X)$. Suppose towards a contradiction that there is $d \in C(N)$ such that $d\in \dcl(\beta(X) \cup \{v\})\setminus \beta(X)$. By o-minimality of $T$ there is a continuous $\Cal L$-$\emptyset$-definable function $f: U \subseteq M^{l+1}\to [0,1]$ and $x \in X^l$ such that $f(\beta(x),d) = v$. Let $g_1,\dots,g_k : M^{l+n} \to M$ be as in Axiom T\ref{axiom:tau} for $f$. Since $u \notin X$ and $C(M)\subseteq X$, we have that $g_i(x,\nu_f(x,u))< u<g_i(x,\tau_f(x,u))$ for $i=1,\dots,k$. By Lemma \ref{lem:nusandtaus} $\gamma(\nu_f(x,u))=\nu_f(\beta(x),v)$ and $\gamma(\tau_f(x,u))=\tau_f(\beta(x),v)$. Since $f(\beta(x),d)=v$, there is $i\in \{1,\dots,k\}$ such that either $v=g_i(\beta(x),\nu_f(\beta(x),v))$ or $v=g_i(\beta(x),\tau_f(\beta(x),v))$. But then for this $i$, $u=g_i(x,\nu_f(x,u))$ or $u=g_i(x,\tau_f(x,u))$, contradicting our assumption on $u$. Hence $\dcl(\beta(X)\cup \{v\})\cap C(N) \subseteq \beta(X)$. Because $\beta$ is a $\Cal L_C^+$-embedding, it now follows easily that $\gamma$ is a $\Cal L_C^+$-embedding and $\dcl(\beta(X)\cup\{v\})\unlhd^+ \Cal N$.
\end{proof}

\begin{proof}[Proof of Theorem \ref{thm:complete}]  Let $\kappa = |\Cal L_C^+|$ and let $\Cal M,\Cal N\models \TT^+$ be such that $\Cal M$ and $\Cal N$ are $\kappa^+$-saturated. Remember that by $K$ we denote the interpretation of the $\Cal L$-constant symbols $c_k$, where $k\in K$. By Axiom T\ref{axiom:c} $K \subseteq C(M)$ and $K \subseteq C(N)$. By Axiom T\ref{axiom:o-minimal} there is a $\Cal L$-isomorphism $\beta$ between $\dcl(\emptyset)\subseteq M$ and $\dcl(\emptyset)\subseteq N$ with $\beta(K)=K$. By Axioms T\ref{axiom:buechi} and T\ref{axiom:buechi2} $\beta \tp_{\Cal L_B}(K) = \tp_{\Cal L_B}(K)$. Note that $Q \subseteq \dcl(\emptyset)$ and $\dcl(\emptyset)$ is archimedean. Therefore $A(M)\cap \dcl(\emptyset)=Q$ and $\dcl(\emptyset)$ is $A$-closed. By Axiom T\ref{axiom:buechi2} $\dcl(\emptyset)$ is closed under $e(a,-)$ for every $a\in Q$. By Proposition \ref{cor:aclosure} $\dcl(\emptyset)\unlhd \Cal M$. By the same argument we get that $\dcl(\emptyset) \unlhd \Cal N$. By Lemma \ref{lem:yspecial} $\beta$ is a $\Cal L_C^*$-isomorphism. As in the proof of Proposition \ref{prop:extendbycplus} we can extend this $\Cal L_C^*$-embedding to a partial $\Cal L_C^+$-isomorphism $\gamma : X \to Y$ such that $X \unlhd^+ \Cal M$ and $Y \unlhd^+ \Cal N$. Because $\TT^+$ has quantifier elimination, we have that $X \equiv \Cal M$ and $Y \equiv \Cal N$. Since $X$ and $Y$ are isomorphic, $X \equiv Y$. Hence $\Cal M \equiv \Cal N$. Thus $\TT^+$ is complete and so is $\TT$.
\end{proof}

\section{Definable sets are Borel}
In this section it will be shown that every set definable in $(\Cal R,K)$ is Borel. The main ingredients of the proof is the quantifier elimination result established in the previous section. Using Fact \ref{fact:landweber} we will establish that the interpretation of every $\Cal L_C^+$-relation symbol and $\Cal L_C^+$-function symbol in $(\Cal R,K)$ is Borel. It then follows easily from Theorem \ref{thm:qe} using elementary results from descriptive set theory that every definable set is Borel.\newline

\noindent We first introduce some new notation we will use. We write $\Cal R_K$ for the $\Cal L_C$-structure $(\Cal R, K, Q, \epsilon)$. As usual we will consider it is a $\Cal L_C^+$-structure. For $q\in Q$, set $K_q:= \{ c \in K \ : \ e(q,c)=c\}$ and set $K_{\operatorname{fin}}:= \bigcup_{q\in Q} K_q$. Note that $K_q$ is $\Cal L_C$-definable over $q$ and $K_{\operatorname{fin}}$ is $\Cal L_C$-$\emptyset$-definable. Also note that $K_q$ is finite for each $q\in Q$ and $K_{\operatorname{fin}}$ is countable. Let $f: Q \to \R$ be a $\Cal L_C$-definable function. We define $\lim_{q\in Q} f(q)$ as the element $x \in [0,1]$ such that for all $\varepsilon >0$ there is $b\in Q$ such that $|f(b')-x|< \varepsilon$ for all $b'\in Q_{>b}$. Obviously, if such $x$ exists, it is unique. We set
\[
\liminf_{q \in Q} f(q) := \lim_{q \in Q} \inf \{ f(b) \ : \ b \in Q_{\geq q}\}.
\]
For $c=(c_1,\dots,c_n) \in K^n$ and $q\in Q$, set $U_{c,q} := [c_1,c_1+q^{-1}]\times \dots \times [c_n,c_n+q^{-1}]$.\newline

\noindent Now fix a continuous $\Cal L$-$\emptyset$-definable function $f: X\subseteq \R^{l+n} \to [0,1]$ such that $X$ is open. Let $g_1,\dots,g_k$ be as in Axiom T\ref{axiom:tau}. For $(x,y)\in \R^{l+1}$, let
\[
l(x,y) := \sup_{d \in K_{\operatorname{fin}}\cap X_x, f(x,d)\leq y} f(x,d), \ r(x,y):= \inf_{d \in K_{\operatorname{fin}}\cap X_x, f(x,d)\geq y} f(x,d).
\]
Because $X$ is open and $K_{\operatorname{fin}}$ is dense in $K$, we directly get
\[
l(x,y) = \sup_{d \in K\cap X_x, f(x,d)\leq y} f(x,d), \ r(x,y) = \inf_{d \in K\cap X_x, f(x,d)\geq y} f(x,d).
\]
Since $K_{\operatorname{fin}}$ is countable, it follows easily that the graphs of $l$ and $r$ are Borel.

\begin{defn} For $q\in Q$, $(x,y)\in \R^{l+1}$, let $D_{q,x,y}\subseteq K_q$ be the set of all $d \in K_q$ such that for all $b\in Q$ there  exists $d'\in K_{\operatorname{fin}} \cap X_x$ such that $d'\in U_{d,q}$ and $0 \leq l(x,y) - f(x,d') < b^{-1}$.
\end{defn}
\noindent Since $Q$ and $K_{\operatorname{fin}}$ are countable and the graph of $l$ is Borel, the set $\{ (q,x,y,c) \ : \ c \in D_{q,x,y}\}$ is Borel as well.

\begin{lem}\label{lem:lg} Let $g : Q \to K_{\operatorname{fin}}$ be such that $g(q) \in D_{q,x,y}$. Then there is $i\in \{1,\dots, k\}$ such that $l(x,y) = g_i(x,\liminf_{q\in Q} g(q))$.
\end{lem}
\begin{proof} Let $c:=\liminf_{q\in Q} g(q)$. Since $K$ is closed, $c\in K$. Suppose towards a contradiction that $l(x,y)\neq g_i(x,c)$ for each $i=1,\dots,k$. Then by Lemma \ref{lem:limitg} there are $a,q\in Q$ such that $|l(x,y) - f(x,d)| > a^{-1}$ for all $d \in B_{q^{-1}}(c)\cap X_x \cap K_{\operatorname{fin}}$. Let $b\in Q$ be such that $b > \max \{2nq,a\}$ and $|c-g(b)| < \frac{q^{-1}}{2}$. Since $g(b) \in D_{b,x,y}$, there is $d \in U_{g(b),b}$ such that $|l(x,y) - f(x,d)|< b^{-1}$. Because $d \in U_{g(b),b}$ and $b>2nq$,
\[
|d -c| < |d-g(b)| + |c-g(b)| < nb^{-1} + \frac{q^{-1}}{2} < q^{-1},
\]
contradicting our assumption on $q$.\end{proof}

\noindent For $m,n\in \N$ with $m< n$, $c\in K^n$ and $q \in Q$ we set
\[
V_{c,q,m} := \{ d=(d_1,\dots,d_n) \in K_q^n \ : \ e(q,c_i) = d_i \wedge e(q,c_{m+1}) > d_{m+1}.\}.
\]

\begin{prop} Let $c=(c_1,\dots,c_n) \in K^n$. The following are equivalent:
\begin{itemize}
\item[(i)] $c=\nu_f(x,y)$
\item[(ii)] $l(x,y)=g_i(x,c)$ for some $i\in \{1,\dots,k\}$ and for $m=0,\dots,n-1$
\[
\forall q\in Q\forall d \in D_{q,x,y}\cap V_{c,q,m} \exists b \in Q_{\geq q} \ U_{d,q} \cap D_{b,x,y}\cap V_{c,b,m} =\emptyset.
\]
\end{itemize}
\end{prop}
\begin{proof} Suppose $c=\nu_f(x,y)$. Towards a contradiction suppose that (ii) fails. Then there are $q\in Q$ and $d=(d_1,\dots,d_n) \in D_{q,x,y}\cap V_{c,q,m}$ such that $U_{d,q} \cap D_{b,x,y}\cap V_{c,b,m} \neq \emptyset$ for all $b\in Q_{\geq q}$. Let $h: Q \to \R$ be the $\Cal L_C$-definable function that maps $b$ to the lexicographically smallest such $d'\in U_{d,q} \cap D_{b,x,y}\cap V_{c,b,m}$ if $b\geq q$ and to $0$ otherwise. Note that the lexicographic minimum in the definition of $h$ exists, because $D_{b,x,y}$ is finite. Set $c'=(c_1',\dots,c_n'):= \liminf_{b\in Q}h(b)$.  Because $h(b)\in V_{c,b,m}$ for $b \in \Q_{\geq q}$, we get that $c_i = c_i'$ for $i=1,\dots, m$ and $c_{m+1} \geq c_{m+1}'$. However, $h(b)\in U_{q,d}$ for each $b\geq q$ and hence so is $c'$. Since $d_{m+1} < e(q,c_{m+1})$, $d_{m+1} +q^{-1} < e(q,c_{m+1})$. Hence $c_{m+1} > c_{m+1}'$. Thus $c'$ is lexicographically smaller than $c$. This contradicts $c=\nu_f(x,y)$, because by Lemma \ref{lem:lg} there is $i\in \{1,\dots, k\}$ such that $l(x,y) = g_i(x,\liminf_{q\in Q} g(q))=g_i(x,c')$.\newline

\noindent Now suppose that $c$ satisfies (ii). By definition of $\nu_f$ it is only left to show that $c$ is lexicographically minimal such that there is $i\in \{1,\dots, k\}$ with $l(x,y) = g_i(x,c)$. Suppose not. Let $c'=(c_1',\dots,c_n') \in K^n$ such that $c'$ is lexicographically smaller than $c$ and there is $i\in \{1,\dots, k\}$ with $l(x,y) = g_i(x,c')$. Let $m < n$ be such that $c_i = c'_i$ for $i=1,\dots,m$ and $c_{m+1} > c_{m+1}'$. Then $e(b,c_i)=e(b,c_i')$ for $i=1,\dots, m$ for all $b\in Q$, and there is $q \in Q$ such that $e(b,c_{m+1})>e(b,c_{m+1}')$ for all $b\in Q$ with $b\geq q$. Thus $e(b,c') \in V_{c,b,m}$ for every $b \geq q$. Since $l(x,y) = g_i(x,c')$ for some $i$, we have that $e(b,c') \in D_{b,x,y}$ for every $b\in Q$. Hence $e(b,c') \in U_{e(q,c'),q} \cap D_{b,x,y}\cap V_{c,b,m}$ for all $b\geq q$. This contradicts (ii).
\end{proof}

\begin{cor}\label{cor:nufboreldef} The graph of $\nu_f$ is Borel.
\end{cor}

\noindent Similarly it can be shown that the graph of $\tau_f$ is Borel. We leave the details to the reader. We are now ready to finish the proof of Theorem B.

\begin{proof}[Proof of Theorem B] First note that the interpretation of $A$ in $\Cal R_K$ is $Q$ and hence countable and Borel. The same is true for $K_{\operatorname{fin}}$. The interpretation of $C$ in $\Cal R_K$ is $K$ and closed, in particular Borel. We will first show that the interpretation of each of the function symbols from $\Cal L_C^+$ is a Borel function. It is enough to check that the graph of each function is Borel. Since $\Cal R$ is o-minimal, the graph of every $\Cal L$-definable function is Borel. It follows immediately from its definition that the graph of $\lambda$ is Borel. By Corollary \ref{cor:nuboreldef} the graph of $\nu$ is Borel. By Corollary \ref{cor:nufboreldef} the graphs of $\nu_f$ and $\tau_f$ are Borel. Since the sets definable from $\Cal L_B$-formula are Borel by Fact \ref{fact:landweber}, every $\Cal L_B$-definable function is Borel. Hence all interpretation of function symbols from $\Cal L_C^+$ are Borel. Again, because sets definable from a $\Cal L_B$-formula or a $\Cal L$-formula are Borel, the interpretation of any predicate symbol from $\Cal L_C^+$ is Borel. Since Borel sets are closed under preimages under Borel functions and under boolean combinations, every set definable by a quantifier-free $\Cal L_C^+$-formula is Borel. By Theorem \ref{thm:qe} every set definable in $\Cal R_K$ is Borel.
\end{proof}

  \bibliographystyle{plain}
  \bibliography{hieronymi}

\begin{thebibliography}{10}

\bibitem{BRW}
Bernard Boigelot, St{\'e}phane Rassart, and Pierre Wolper.
\newblock On the expressiveness of real and integer arithmetic automata
  (extended abstract).
\newblock In {\em Proceedings of the 25th International Colloquium on Automata,
  Languages and Programming}, ICALP '98, pages 152--163, London, UK, UK, 1998.
  Springer-Verlag.

\bibitem{Buchi}
J.~Richard B{\"u}chi.
\newblock On a decision method in restricted second order arithmetic.
\newblock In {\em Logic, {M}ethodology and {P}hilosophy of {S}cience ({P}roc.
  1960 {I}nternat. {C}ongr .)}, pages 1--11. Stanford Univ. Press, Stanford,
  Calif., 1962.

\bibitem{Landweber}
J.~Richard B{\"u}chi and Lawrence~H. Landweber.
\newblock Definability in the monadic second-order theory of successor.
\newblock {\em J. Symbolic Logic}, 34:166--170, 1969.

\bibitem{Lou}
Lou van~den Dries.
\newblock {$T$}-convexity and tame extensions. {II}.
\newblock {\em J. Symbolic Logic}, 62(1):14--34, 1997.

\bibitem{densepairs}
Lou van~den Dries.
\newblock Dense pairs of o-minimal structures.
\newblock {\em Fund. Math.}, 157(1):61--78, 1998.

\bibitem{tametop}
Lou van~den Dries.
\newblock {\em Tame topology and o-minimal structures}, volume 248 of {\em
  London Mathematical Society Lecture Note Series}.
\newblock Cambridge University Press, Cambridge, 1998.

\bibitem{LouLew}
Lou van~den Dries and Adam~H. Lewenberg.
\newblock {$T$}-convexity and tame extensions.
\newblock {\em J. Symbolic Logic}, 60(1):74--102, 1995.

\bibitem{Falconer}
Kenneth Falconer.
\newblock {\em Fractal geometry}.
\newblock John Wiley \& Sons, Inc., Hoboken, NJ, second edition, 2003.
\newblock Mathematical foundations and applications.

\bibitem{FHM}
Antongiulio Fornasiero, Philipp Hieronymi, and Chris Miller.
\newblock A dichotomy for expansions of the real field.
\newblock {\em Proc. Amer. Math. Soc.}, 141(2):697--698, 2013.

\bibitem{FKMS}
Harvey Friedman, Krzysztof Kurdyka, Chris Miller, and Patrick Speissegger.
\newblock Expansions of the real field by open sets: definability versus
  interpretability.
\newblock {\em J. Symbolic Logic}, 75(4):1311--1325, 2010.

\bibitem{Miller-fast}
Harvey Friedman and Chris Miller.
\newblock Expansions of o-minimal structures by fast sequences.
\newblock {\em J. Symbolic Logic}, 70(2):410--418, 2005.

\bibitem{discrete}
Philipp Hieronymi.
\newblock Defining the set of integers in expansions of the real field by a
  closed discrete set.
\newblock {\em Proc. Amer. Math. Soc.}, 138(6):2163--2168, 2010.

\bibitem{discrete2}
Philipp Hieronymi.
\newblock Expansions of subfields of the real field by a discrete set.
\newblock {\em Fund. Math.}, 215(2):167--175, 2011.

\bibitem{HT}
Philipp Hieronymi and Michael Tychonievich.
\newblock Interpreting the projective hierarchy in expansions of the real line.
\newblock {\em Proc. Amer. Math. Soc.}, 142(9):3259--3267, 2014.

\bibitem{Kechris}
Alexander~S. Kechris.
\newblock {\em Classical descriptive set theory}, volume 156 of {\em Graduate
  Texts in Mathematics}.
\newblock Springer-Verlag, New York, 1995.

\bibitem{automata}
Bakhadyr Khoussainov and Anil Nerode.
\newblock {\em Automata theory and its applications}, volume~21 of {\em
  Progress in Computer Science and Applied Logic}.
\newblock Birkh\"auser Boston, Inc., Boston, MA, 2001.

\bibitem{lms}
Jean-Marie Lion, Chris Miller, and Patrick Speissegger.
\newblock Differential equations over polynomially bounded o-minimal
  structures.
\newblock {\em Proc. Amer. Math. Soc.}, 131(1):175--183 (electronic), 2003.

\bibitem{Marker}
David Marker.
\newblock {\em Model theory}, volume 217 of {\em Graduate Texts in
  Mathematics}.
\newblock Springer-Verlag, New York, 2002.
\newblock An introduction.

\bibitem{McNaughton}
Robert McNaughton.
\newblock Testing and generating infinite sequences by a finite automaton.
\newblock {\em Information and Control}, 9:521--530, 1966.

\bibitem{growth}
Chris Miller.
\newblock Exponentiation is hard to avoid.
\newblock {\em Proc. Amer. Math. Soc.}, 122(1):257--259, 1994.

\bibitem{Miller-tame}
Chris Miller.
\newblock Tameness in expansions of the real field.
\newblock In {\em Logic {C}olloquium '01}, volume~20 of {\em Lect. Notes Log.},
  pages 281--316. Assoc. Symbol. Logic, Urbana, IL, 2005.

\bibitem{Miller-iteration}
Chris Miller and James Tyne.
\newblock Expansions of o-minimal structures by iteration sequences.
\newblock {\em Notre Dame J. Formal Logic}, 47(1):93--99, 2006.

\bibitem{Rabin}
Michael~O. Rabin.
\newblock Decidability of second-order theories and automata on infinite trees.
\newblock {\em Trans. Amer. Math. Soc.}, 141:1--35, 1969.

\bibitem{pfaffian}
Patrick Speissegger.
\newblock The {P}faffian closure of an o-minimal structure.
\newblock {\em J. Reine Angew. Math.}, 508:189--211, 1999.

\bibitem{Tyne}
James Tyne.
\newblock {\em $T$-levels and $T$-convexity}.
\newblock PhD thesis, University of Illinois at Urbana-Champaign, 2003.

\bibitem{TyneJSL}
James Tyne.
\newblock {$T$}-height in weakly o-minimal structures.
\newblock {\em J. Symbolic Logic}, 71(3):747--762, 2006.

\end{thebibliography}

\end{document}